\documentclass[10pt,a4paper, twoside]{amsart}
\usepackage[T1]{fontenc}
\usepackage[latin1]{inputenc}
\usepackage[english]{babel}
\usepackage{amsmath}
\usepackage{amsthm}
\usepackage{amssymb}
\usepackage{amsfonts}
\usepackage{amsxtra}
\usepackage{enumerate}
\usepackage{verbatim}
\usepackage{color}
\usepackage[mathscr]{eucal}
\usepackage{enumerate}

\DeclareMathOperator*{\supp}{supp}
\DeclareMathOperator*{\conv}{conv}
\DeclareMathOperator*{\diam}{diam}

\DeclareMathOperator*{\dist}{dist}

\DeclareMathOperator*{\pos}{pos}

\DeclareMathOperator*{\fil}{fill}
\DeclareMathOperator*{\err}{err}

\DeclareMathOperator*{\Time}{Time}
\DeclareMathOperator*{\Light}{Light}
\DeclareMathOperator*{\std}{std(g_R)}
\DeclareMathOperator*{\Lor}{Lor}

\newcommand{\N}{\mathbb N}
\newcommand{\Z}{\mathbb Z}
\newcommand{\Q}{\mathbb Q}
\newcommand{\R}{\mathbb R}

\newcommand{\C}{\mathbb C}

\newcommand{\T}{\mathfrak{T}}
\newcommand{\e}{\varepsilon}

\newtheorem{theoremloc}{Theorem}
\newtheorem{definitionloc}[theoremloc]{Definition}
\newtheorem{propositionloc}[theoremloc]{Proposition}
\newtheorem{lemmaloc}[theoremloc]{Lemma}
\newtheorem{corollaryloc}[theoremloc]{Corollary}
\newtheorem*{theoremlocstar}{Theorem}{\bf}{\it}
\newtheorem*{propositionlocstar}{Proposition}{\bf}{\it}
\newtheorem{example}{Example}
\newtheorem{remark}{Remark}
\newtheorem{note}{Note}

\begin{document}

\title{Class A Spacetimes}

\author[Stefan Suhr]{Stefan Suhr}
\address{Fachbereich Mathematik, Universit\"at Hamburg}
\email{stefan.suhr@math.uni-hamburg.de}

\date{\today}

\begin{abstract}
We introduce class A spacetimes, i.e. compact vicious spacetimes $(M,g)$ such that the Abelian cover $(\overline{M},\overline{g})$ is globally hyperbolic. We 
study the main properties of class A spacetimes using methods similar to those introduced in \cite{sul} and \cite{bu}. As a consequence we are able to 
characterize manifolds admitting class A metrics completely as mapping tori. Further we show that the notion of class A spacetime is equivalent to that of SCTP 
(spacially compact time-periodic) spacetimes as introduced in \cite{ga84}. The set of class A spacetimes is shown to be open in the $C^0$-topology on the set of 
Lorentzian metrics. As an application we prove a coarse Lipschitz property for the time separation of the Abelian cover. This coarse Lipschitz property is an 
essential part in the study of Aubry-Mather theory in Lorentzian geometry.


\end{abstract}

\maketitle

\section{Introduction}\label{caus}

The theory of compact Lorentzian manifolds is in large parts terra incognita. In opposition to Riemannian geometry, Lorentzian geometry is focused on noncompact 
manifolds, for well known reasons motivated by physical intuition in general relativity. The situation with compact Lorentzian manifolds is vague to the extent that 
there is no well established large subclass of compact Lorentzian manifolds with well understood geometric features. It is the purpose of these notes to propose 
one such class (class A) and study some of its properties. We point out that, though class A spacetimes justify a study by themselves, the main application for these 
spacetimes will be the study of homologically maximizing causal geodesics (Aubry-Mather theory) in subsequent publications (\cite{suh110}, \cite{suh111}). 

A compact spacetime $(M,g)$ is said to be {\it class A} if $(M,g)$ is vicious and the Abelian cover is globally hyperbolic (see section \ref{lorba} for definitions).
Class A spacetimes appear for the first time for $2$-dimensional compact manifolds in \cite{suh102}. Though the definition of class A therein differs from the one 
above, \cite{suh102} shows that this definition and the definition in \cite{suh102} are equivalent for $\dim M=2$. \cite{suh102} contains the 
justification for the name `class A' as well. Therein a distinction is made between class A and class B spacetimes. It is by far not clear what a good definition of 
class B spacetimes should be for $\dim M\ge 3$. 

First examples of class A spacetimes are flat Lorentzian tori, i.e. quotients of Minkowski space by a cocompact lattice. Other known examples are spacetime 
structures on $2$-tori admitting either a timelike or spacelike conformal Killing vector field (\cite{suh102}).

It turns out that class A spacetimes are closely related to the known class of SCTP (spatially compact time periodic) spacetimes introduced in \cite{ga84} (see 
definition \ref{DSCTP} and theorem \ref{stab2}). SCTP spacetimes $(M',g')$ are globally hyperbolic spacetimes with a compact Cauchy hypersurface $\Sigma$ 
and an isometric $\Z$-action (translation into the future) by isometries $\psi_n \colon M'\to M'$ such that $\Sigma_{n+1}:=\psi_{n+1}(\Sigma)\subset 
I^+(\Sigma_{n})$ and the orbit of every $x\in M'$ intersects $I^+(x)$. Then the time orientable quotient of a SCTP spacetime is class A and every class A spacetime 
is covered by a SCTP spacetime (theorem \ref{stab2}). Recent work in mathematical physics revealed the relevance of SCTP (and with it class A) spacetime to 
general relativity (\cite{fihai}). 

The definition of class A spacetimes in terms of causality conditions yields surprising restrictions on the topological and geometric structure of these spacetimes. 
The main result of these notes is theorem \ref{stab2}. Theorem \ref{stab2} has three important corollaries (corollary \ref{C1}, \ref{stab} and \ref{C2}). 

Corollary \ref{C1} gives a precise characterization of manifolds that admit class A metrics. Like in the case of globally hyperbolic spacetimes, existence of class A 
spacetime structures induce strong restrictions on the topology of $M$, i.e. there exists a class A metric in $\Lor(M)$ iff $M$ is diffeomorphic to a mapping torus. 
Note that this result represents a compact version of the global splitting theorem for globally hyperbolic spacetimes (\cite{ger1},\cite{besa3}) for compact 
spacetimes.

Corollary \ref{stab} states that the set of class A metrics (i.e. Lorentzian metrics on $M$ such that $(M,g)$ is class A) is open in the $C^0$-uniform topology on the 
space of Lorentzian metrics $\Lor(M)$ on $M$. This represents a uniform version of theorem 12 in \cite{ger1}: For any globally hyperbolic spacetime $(M,g)$ there 
exists an open neighborhood $U$ of $g$ in $Lor(M)$, equipped with the fine $C^0$-topology, such that any Lorentzian metric $g_1\in U$ is globally hyperbolic 
as well. Note that one cannot use fine neighborhoods as in \cite{ger1} for $\overline g:=\overline{\pi}^\ast g$ directly, since the topology induced on $\Lor(M)$ 
by the canonical projection $\overline{\pi}\colon \overline{M}\to M$ is finer than the uniform topology on $\Lor(M)$, and therefore $\overline g$ might be the only 
periodic Lorentzian metric in $U$.

Finally with corollary \ref{C2} we prove that any SCTP spacetime is isometric to a ``standard SCTP'' spacetime. Consider spacetimes $(\R\times \Sigma, -f^2dt^2+ 
g_\Sigma)$ where $f$ is a positive function on $\R\times \Sigma$ and $g_\Sigma$ is a Riemannian metric on $\Sigma$ i.g. depending on the $t$-coordinate as 
well. These spacetimes are called {\it standard SCTP} spacetime if there exists an isometry $\psi \colon \R\times\Sigma\to \R\times \Sigma$ of the form $\psi(t,x)=
(t+1,\phi(x))$ or $\psi(t,x)=(-t+1,\phi(x))$ for some diffeomorphism $\phi \colon \Sigma \to \Sigma$. Corollary \ref{C2} represents an analogue to the Lorentzian 
splitting theorem for globally hyperbolic spacetimes with a group action. 

The proof of theorem \ref{stab2} incorporates several different constructions and methods, e.g. Sullivan's structure cycles (\cite{sul}, see appendix \ref{cones}), a 
generalization of a methods introduced by D. Yu Burago (\cite{bu}) and the construction of the homological timecone $\T$ (see section \ref{sec2}) the 
Schwartzmann's cone for $1$-cycles with support in the future pointing vectors. The homological timecone can be seen as an asymptotic (i.e. stable) version 
of the causality relations in the Abelian cover, much in the same way the stable norm on $H_1(M,\R)$ (\cite{gromov}, 4.19) can be seen as an asymptotic version of 
the Riemannian distance function on the Abelian  cover. Example \ref{E1} shows that the result of theorem \ref{stab2} is in some respect optimal, i.e. the 
viciousness assumption cannot be dropped. 

The second main result, theorem \ref{T16}, claims the coarse Lipschitz property of the time separation (Lorentzian distance, see section \ref{lorba}) of the 
Abelian cover of a class A spacetime. The Lipschitz continuity of the time separation has received very little attention in the literature 
so far. It made a short appearance in connection with the Lorentzian version of the Cheeger-Gromoll splitting theorem (\cite{es}, \cite{gaho}). 
The idea we employ here is different from the approaches before and is based on so-called cut-and-paste arguments commonly used 
in Aubry-Mather theory (\cite{ba},\cite{ma}). 

The text is structured as follows: In section \ref{GNN} we collect the necessary notions from Lorentzian and Riemannian geometry and set the global notation. In 
section \ref{dipl} we review previous work on globally conformally flat tori. In section \ref{sec2} we introduce the stable time cone $\T$, the homological equivalent 
of the causal future and discuss the main results and examples mentioned so far. Finally sections \ref{proof} and \ref{rc} contain the proofs of theorem \ref{stab2} 
and theorem \ref{T16}.

\section{Geometric Notions and Notation}\label{GNN}

\paragraph{Notation}\label{not}

Throughout the article we consider smooth manifolds only. $\mathcal{D}(M',M)$ denotes the group of deck transformations for a regular cover $\pi'\colon 
M'\to M$. By $\overline{M}$ we denote the quotient of the universal cover $\widetilde{M}$ by the commutator group of $\pi_1(M)$, i.e. $\overline{M}\cong
\widetilde{M}/[\pi_1(M),\pi_1(M)]$. $\overline{M}$ will be called the {\it Abelian cover} of $M$. Denote with $\overline{\pi}$ the canonical projection of $\overline{M}$ 
to $M$. Further we denote with $H_1(M,\mathbb{Z})_\mathbb{R}$ the image of the natural map $H_1(M,\mathbb{Z})\to H_1(M,\mathbb{R})$. We denote the action 
of $\mathcal{D}(\overline{M},M)$ by $+$, i.e. 
$$(k,x)\in \mathcal{D}(\overline{M},M)\times \overline{M}\mapsto x+k\in \overline{M}.$$

\subsection{Riemannian structures}\label{GMT}

We will need the concept of {\it rotation vectors} from \cite{ma}. Let $k_1,\ldots ,k_b$ $(b:=dim\; H_1(M,\mathbb{R}))$ be a basis of $H_1(M,\mathbb{R})$ consisting 
of integer classes, and $\alpha_1,\ldots ,\alpha_b$ the dual basis with representatives $\omega_1,\ldots ,\omega_b$. For two points $x,y \in \overline{M}$ we 
define the {\it difference} $y-x\in H_1(M,\mathbb{R})$ via a $C^1$-curve $\gamma\colon [s,t]\to \overline{M}$ connecting $x$ and $y$, by 
  $$\langle \alpha_i, y-x\rangle :=\int_\gamma \pi^*\omega_i$$
for all $i\in \{1,\ldots,b\}$. The rotation vector of  $\gamma$ as well as of $\overline{\pi}\circ \gamma$ is defined as 
$$\rho(\gamma):=\rho(\overline{\pi}\circ\gamma):=\frac{1}{t-s}(y-x).$$
Note that the map $(x,y)\mapsto y-x$ is i.g. not surjective. But we know that the convex hull of the image is $H_1(M,\mathbb{R})$. Just observe that by our 
choice of classes $\alpha_i$ we know that every $k\in H_1(M,\mathbb{Z})_\mathbb{R}$ is the image of $(x,x+k')$ for every $x\in \overline{M}$ where 
$k'\in H_1(M,\mathbb{Z})\mapsto k\in H_1(M,\mathbb{Z})_\mathbb{R}$ under the natural map.

We choose a Riemannian metric $g_R$ on $M$ arbitrary but fixed once and for all. We denote the distance function relative to $g_R$ by $dist$ and the metric 
balls of radius $r$ around $p\in M$ with $B_r(p)$. The metric $g_R$ induces a norm on every tangent space of $M$, which we denote by $|.|$, i.e. $|v|:=
\sqrt{g_R(v,v)}$ for all $v\in TM$. Further we denote by $T^{1,R}M$ the unit tangent bundle of $(M,g_R)$. For convenience of notation we denote the lift of $g_R$ 
to $\overline{M}$, and all objects associated to it, with the same letter. Set 
$$\diam(M,g_R):=\max_{p\in M} \min_{k\in H_1(M,\mathbb{Z})\setminus\{0\}}\{\dist(\overline{p},\overline{p}+k)|\; \overline{p}\in 
\overline{\pi}^{-1}(p)\}$$
the homological diameter of $(M,g_R)$.

We will constantly employ the following theorem. 
\begin{theoremloc}[\cite{bu}, \cite{kett}]\label{3.1}
Let $(M,g_R)$ be a compact Riemannian manifold. Then there exists a unique norm $\|.\|\colon H_1(M,\mathbb{R})\to \mathbb{R}$ and a constant 
$\std<\infty$ such that    
$$|\dist(x,y)-\|y-x\||\le \std$$
for any $x,y\in \overline{M}$.
\end{theoremloc}

$\|.\|$ is called the stable norm of $g_R$ on $H_1(M,\mathbb{R})$. The distance function on $H_1(M,\mathbb{R})$ relative to $\|.\|$ is written as $\dist_{\|.\|}$. By 
$\|.\|^\ast$ we denote its dual norm on $H^1(M,\mathbb{R})$.

\subsection{Lorentzian Geometry}\label{lorba}

Let $(M,g)$ be a spacetime, i.e. a Lorentzian manifold such there exists $X\in \Gamma(TM)$ with $g(X,X)<0$ (we use the sign convention $(-,+,\ldots ,+)$). $X$ 
is called a {\it time-orientation}. Denote by $[g]$ the conformal class of the Lorentzian metric $g$ sharing the same time-orientation, i.e. all Lorentzian metrics $g'$ 
such that there exists $u\in C^\infty (M)$ with $g'=e^u g$ and the time orientations of $g$ and $g'$ coincide. For $[g]$ define the sets
$$\Time(M,[g]):=\{\text{future pointing timelike vectors in $(M,g)$}\}$$
and
$$\Light(M,[g]):=\{\text{future pointing lightlike vectors in $(M,g)$}\}.$$
Both $\Time(M,[g])$ and $\Light(M,[g])\setminus \{\text{zero section}\}$ are smooth fibre bundles over $M$. Denote by $\Time(M,[g])_p$ and $\Light(M,[g])_p$ the 
fibres of $\Time(M,[g])$ and $\Light(M,[g])$ over $p\in M$, respectively. For $\e>0$ we define
$$\Time(M,[g])^\e:=\{v\in \Time(M,[g])|\; \dist(v,\Light(M,[g])\ge \e|v|\}.$$
$\Time(M,[g])^\e$ is a smooth fibre bundle as well with fibre $\Time(M,[g])^\e_p$ over $p\in M$. The fibres are convex for every $p\in M$ according to the following 
lemma and corollary.
 
\begin{lemmaloc}\label{L100}
Let $(V,|.|)$ be a finite-dimensional normed vector space and $\mathfrak{K}\subset V$ a convex set with $\mathfrak{K}\neq V$. 
Then the function $v\in \mathfrak{K}\mapsto \dist_{|.|}(v,\partial\mathfrak{K})$ is concave.
\end{lemmaloc}

The proof is an exercise in convex geometry. See \cite{cg1} theorem 1.10 for a proof in the more general case that $(M,g_R)$ 
is Riemannian manifold of nonnegative curvature.

If $\mathfrak{K}$ is a convex cone we know that $v\in \mathfrak{K}\mapsto \dist_{|.|}(v,\partial\mathfrak{K})$ is 
positively homogenous of degree one, i.e. $\dist_{|.|}(\lambda v,\partial \mathfrak{K})=
\lambda \dist_{|.|}(v,\partial \mathfrak{K})$ for all $\lambda \ge 0$. 
Lemma \ref{L100} and the positive homogeneity then imply
$$\dist\nolimits_{|.|}(v+w,\partial \mathfrak{K})\ge \dist\nolimits_{|.|}(v,\partial\mathfrak{K})+
\dist\nolimits_{|.|}(w,\partial\mathfrak{K}).$$

\begin{corollaryloc}\label{C100}
Let $\mathfrak{K}\neq V$ be a convex cone and $\e>0$. The cones 
$$\mathfrak{K}_\e:=\{v\in \mathfrak{K}|\; \dist\nolimits_{|.|}(v,\partial\mathfrak{K})\ge \e |v|\}$$ 
are convex for all $\e>0$. 
\end{corollaryloc}

Finally we give the definitions for the notions from causality theory used in these notes. For an overview of causality theory of Lorentzian manifolds see \cite{bee}.
A piecewise smooth curve $\gamma\colon [a,b]\to M$ is called {\it future pointing (future pointing timelike)} if all left and right sided tangents are future pointing
(future pointing timelike). Then the {\it time separation} $d(p,q)$ of two points $p,q\in M$ is given as the supremum of the Lorentzian arclength of future pointing 
curves from $p$ to $q$. Here the supremum over the empty set is defined as $0$. Note that every future pointing curve admits a Lipschitz parameterization and is 
therefore rectifiable. 

A spacetime is called {\it vicious} if every point lies on a timelike loop. Equivalently one can suppose that the chronological past and future of every point are equal 
to the entire manifold. A spacetime $(M,g)$ is {\it globally hyperbolic} if the there exists a subset $S\subset M$ such that every inextendible timelike curves 
intersects $S$ exactly once.

\section{Preceding Work }\label{dipl}

In this section we shortly review a preceding study (\cite{su}) with results that served as a motivation for the present article. The part of \cite{su} relevant to this text 
is concerned with the problem of Lipschitz continuity of the time separation in the Abelian cover of a globally conformally flat Lorentzian torus. Note that globally 
conformally flat Lorentzian tori are trivially of class A. Before we formulate the results, we explain the precise setup.

Consider a real vector space $V$ of dimension $m<\infty$ and $\langle .,.\rangle_1$ a nondegenerate symmetric bilinear form on $V$ with signature 
$(-,+,\ldots ,+)$. Further let $\Gamma\subseteq V$ be a co-compact lattice and $f\colon V \to (0,\infty)$ a smooth and $\Gamma$-invariant function. The Lorentzian 
metric $\overline{g}:= f^2\langle .,.\rangle_1$ then descends to a Lorentzian metric on the torus $V/\Gamma$. Denote the induced Lorentzian metric by $g$. 
Choose a time orientation of $(V,\langle .,.\rangle_1)$. This time orientation induces a time orientation on  $(V/\Gamma,g)$ as well. Note that $(V/\Gamma,g)$ is 
vicious and the universal cover $(V,\overline{g})$ is globally hyperbolic. According to \cite{rosa1} proposition 2.1, $(V/\Gamma,g)$ is geodesically complete in all 
three causal senses. Fix a norm $\|.\|$ on $V$ and denote the dual norm by $\|.\|^\ast$. Note that $\|.\|$ induces a metric on $V/\Gamma$. Further denote by $\T$ 
the positively oriented causal vectors of $(V,\langle .,.\rangle_1)$. 

For $\e>0$ set $\T_\e:= \{v\in \T|\, \dist(v,\partial \T)\ge \e\|v\|\}$. Choose an orthonormal basis $\{e_1,\ldots ,e_m\}$ of $(V,\langle .,.\rangle_1)$. Note that the 
translations $x\mapsto x+v$ are conformal diffeomorphisms of $(V,\overline{g})$ for all $v\in V$. Then the $\overline{g}$-orthogonal frame field $x\mapsto 
(x,(e_1,\ldots ,e_m))$ on $V$ descends to a $g$-orthogonal frame field on $V/\Gamma$. Relative to this identification of $V\cong TV_p$ follows $\T=
\Time(V,[\overline{g}])_p\cup \Light(V,[\overline{g}])_p$ and $\T_\e=\Time(V,[\overline{g}])^\e_p$. 

\cite{su} contains the following compactness result for future pointing maximizers in $(V/\Gamma,g)$.
\begin{theoremloc}[\cite{su}]\label{TD21}
For every $\e>0$ there exists $\delta>0$ such that 
$$\dot\gamma(t)\in \T_\delta$$
for all future pointing maximizers $\gamma\colon I \to V/\Gamma$ with $\dot\gamma(t_0)\in \T_\e$ for some $t_0\in I$ and all $t\in I$.
\end{theoremloc}
Theorem \ref{TD21} has the following immediate consequence.
\begin{corollaryloc}[\cite{su}]
Let $\e>0$. Then any limit curve of a sequence of future pointing maximizers $\gamma_n\colon I_n\to V/\Gamma$ with $\dot{\gamma}_n(t_n)\in \T_\e$, for some 
$t_n\in I_n$, is timelike. 
\end{corollaryloc}
The author then deduces, following \cite{gaho}, the Lipschitz continuity of the time separation $d$ of $(V,\overline{g})$ on $\{(p,q)\in V\times V|\; q-p\in \T_\e\}$ for 
every $\e>0$. Using the standard argument that local Lipschitz continuity with a fixed Lipschitz constant implies Lipschitz continuity, one obtains the following 
theorem.
\begin{theoremloc}[\cite{su}]
For all $\e>0$ there exists $L=L(\e)<\infty$ such that the time separation $d$ of $(V,\overline{g})$ is $L$-Lipschitz on $\{(x,y)\in V\times V|\, y-x\in \T_\e\}$.
\end{theoremloc}

\section{Causality Properties of Class A Spacetimes}\label{sec2}

Recall the definition of class A spacetimes from the introduction.
\begin{definitionloc}\label{D1}
A compact spacetime $(M,g)$ is of class A if $(M,g)$ is vicious and the Abelian cover $\overline\pi \colon (\overline M,\overline g)\rightarrow (M,g)$ is globally 
hyperbolic. We call a metric $g\in \Lor(M)$ class A iff $(M,g)$ is class A.
\end{definitionloc}

We begin the discussion of class A spacetimes with an example of naturally appearing spacetimes that give rise to a class A spacetime.

\begin{example}\label{E0}
Let $(N,g_N)$ be a closed (i.e. compact with empty boundary) Riemannian manifold and $\phi\colon (N,g_N)\to (N,g_N)$ an isometry. 
Consider the globally hyperbolic spacetime $(\R\times N, g':=-dt^2+\pi_N^\ast g_N)$ where $\pi_N\colon \R\times N\to N$ denotes the natural projection. 
The group $G\cong \Z$ generated by $\psi\colon (t,x)\to (t+1,\phi(x))$ acts free and properly discontinuously by time orientation preserving isometries on 
$(\R\times N,g')$. Consequently the quotient $M:=(\R\times N)/G$ carries naturally a time orientable Lorentzian metric $g$. Further the quotient is compact. We 
claim that the spacetime $(M,g)$ is class A. 

First we will prove that $(M,g)$ is vicious. Let $(t,x)\in \R\times N$. We will show that $\psi^n(t,x)\in I^+((t,x))$ for some $n\in\N$. The projection of a timelike
curve between $(t,x)$ and $\psi^n(t,x)$ then yields a timelike loop around $G(t,x)\in M$. By construction we know that $\{t+1\}\times N\cap I^+((t,x))=
\{t+1\}\times B_1(x)$ for all $(t,x)\in \R\times N$, where $B_1(x)$ denotes the ball of radius $1$ and center $x$ in $(N,g_N)$. This follows from the fact that for any 
curve $\gamma \colon [t,t+1]\to N$ with $\gamma(t)=x$ and $|\dot{\gamma}|< 1$, the curve $\gamma'(s):=(s,\gamma(s))$ is timelike in $(\R\times N,g')$.
By induction we get $\{t+n\}\times N\cap I^+((t,x))=\{t+n\}\times B_n(x)$. Since $\diam(N,g_N)<\infty$ there exists a $n_0\in \N$ such that $\{t+n_0\}\times N
\subseteq I^+((t,x))$. Thus we have $\psi^{n_0}(t,x)\in I^+((t,x))$ and $(M,g)$ is vicious. 

To show that the Abelian Lorentzian cover $(\overline{M},\overline{g})$ of $(M,g)$ is globally hyperbolic we will show that there exists a closed $1$-form $\omega$
on $M$ such that $\ker \omega$ is a spacelike hyperplane in any tangent space of $M$. The lift of $\omega$ to $\overline{M}$ then is the differential of a Cauchy 
temporal function $\tau_\omega \colon \overline{M}\to \R$. This is well known to be equivalent to global hyperbolicity.

Consider the natural projection $\pi_\R\colon \R\times N\to \R$. The differential $d\pi_\R$ is $G$-invariant and $\ker d\pi_\R$ is spacelike everywhere. Therefore 
$d\pi_\R$ induces a closed $1$-form $\omega$ on $M$. Then the pullback of $\omega$ to $\overline{M}$ is the differential of a temporal function $\tau$, i.e.
$\ker d\tau$ is spacelike everywhere. It remains to show that $\tau$ is a Cauchy temporal function, i.e. $\tau\circ\gamma\colon I\to \R$ is onto for all inextendible 
causal curves $\overline{\gamma}\colon I\to \overline{M}$. Let $\overline{\gamma}\colon I\to \overline{M}$ be an inextendible causal curve in $(\overline{M},
\overline{g})$. Projection of $\overline{\gamma}$ to $M$ and lifting to $\R\times N$ yields a inextendible causal curve $\gamma\colon I\to \R\times N$. 
Consequently we know that $\pi_\R\circ \gamma\colon I\to \R$ is onto since $\pi_\R$ is obviously a Cauchy temporal function on $\R\times N$. By construction we 
have $\tau\circ \overline{\gamma}\equiv \pi_\R\circ \gamma$ up to a constant. Therefore we arrive at the conclusion that $\tau$ is Cauchy. This finishes the proof 
that $(M,g)$ is class A.

Note that the Cauchy hypersurfaces in $(\R\times N,-dt^2+\pi^\ast_N g_N)$ are compact. In fact it can be shown that any Cauchy hypersurface of $(\R\times N,
-dt^2+\pi^\ast_N g_N)$ is diffeomorphic to $N$. We will see below (theorem \ref{stab2}) that the existence of a covering manifold of a compact vicious spacetime 
with a compact Cauchy hypersurface is equivalent to the class A condition. 
\end{example}

\begin{example}[Counterexample]\label{E1} 
A natural question appearing at this point is, if whether any compact spacetime with a globally hyperbolic covering space admits a covering space with a compact 
Cauchy hypersurface. Spacetimes with a compact Cauchy hypersurface are called {\it spatially closed}. In this paragraph we will construct a spacetime structure on 
the $3$-torus with globally hyperbolic Abelian covering but  no spatially closed globally hyperbolic covering, thus showing that viciousness is essential.

Consider $\R^3$ with the canonical coordinates $\{x,y,z\}$. Denote with $\overline{T}_i:=x^{-1}(i)$ for $i=1,\ldots,6$. Choose a $7\cdot\Z^3$-invariant Lorentzian 
metric $\overline{g}$ on $\R^3$ subject to the following conditions:
\begin{itemize}
\item[(i)] $\overline{g}|_{\overline{T}_1+(7\Z) e_1}=\overline{g}|_{\overline{T}_4+(7\Z) e_1}=(dx+dz)dx+dy^2$,
\item[(ii)] $\overline{g}|_{\overline{T}_3+(7\Z) e_1}=\overline{g}|_{\overline{T}_6+(7\Z) e_1}=(dx-dz)dx+dy^2$,
\item[(iii)] $\overline{g}|_{\overline{T}_2+(7\Z) e_1}=-dy dz+dx^2$,
\item[(iv)] $\overline{g}|_{\overline{T}_5+(7\Z) e_1}=dydz+dx^2$ and
\item[(v)] $\ker dz_p$ is spacelike for all $p\notin (\overline{T}_2\cup\overline{T}_5)+7\Z e_1$.
\item[(vi)] $(\R^3,\overline{g})$ contains a timelike periodic curve $\overline{\gamma}\colon [0,1]\to \R^3$ with 
$\overline{\gamma}(1)-\overline{\gamma}(0)=7e_3$ 
\end{itemize}
Since $\R^3$ is simply connected we can choose a time-orientation for $(\R^3,\overline{g})$. Choose the time-orientation such that $dz$ is nonnegative on future 
pointing vectors. Note that by condition (v) the real number $dz(v)$ is either positive or negative for every nonspacelike vector $v\in T\R^3$ except for 
$v=\partial_y$ and $\pi_{T\R^3}(v)\in \overline{T}_2$ or $v=-\partial_y$ and $\pi_{T\R^3}(v)\in \overline{T}_5$.

We can choose $\e>0$ such that $\tau_1\colon \R^3\to \R$, $p\mapsto \e y(p)+z(p)$ is a Cauchy temporal function for $x(p)\in [-1,4]+7\Z$ and $\tau_2\colon \R^3
\to \R$, $p\mapsto -\e y(p)+z(p)$ is a Cauchy temporal function for $x(p)\in [3,8]+7\Z$. Therefore there exists $\e'>0$ such that $|d\tau_1(v)|$ or $|d\tau_2(v)|\ge 
\e'|v|$ for all nonspacelike vectors $v\in T\R^3_p$. We know that the existence of Cauchy temporal functions is sufficient for global hyperbolicity and thus we see 
that $([-1,4]+7\Z,\overline{g})$ and $([5,8]+7\Z,\overline{g})$ are globally hyperbolic. Note that any future pointing curve starting in $x^{-1}([-1,4])$ can never leave 
$x^{-1}([-1,4])$. The same holds for future pointing curves starting in $x^{-1}([3,8])$. Together with the periodicity of $\overline{g}$, these observations imply that 
$(\R^3,\overline{g})$ is globally hyperbolic.

Since we have chosen $\overline{g}$ invariant under translations in $7\cdot\Z^3$, it descends to a Lorentzian metric $g$ on $T^3:=\R^3/(7\cdot \Z^3)$. Note 
$(T^3,g)$ is time-orientable but not vicious (recall the argument that future pointing curves can never leave $x^{-1}([-1,4])$). 

Now assume that there exists a spatially closed cover $\pi'\colon (Z,g')\to(T^3,g)$ with compact Cauchy hypersurface $\Sigma$. Any lift $\overline{\Sigma}$ of 
$\Sigma$ to $\R^3$ has to be a Cauchy hypersurface of $(\R^3,\overline{g})$ (\cite{gaha}). With \cite{besa3} we can assume that $\Sigma$ is spacelike. Note that 
$(\overline{T}_2,\overline{g}|_{\overline{T}_2})$ and $(\overline{T}_5,\overline{g}|_{\overline{T}_5})$ are Lorentzian submanifolds of $(\R^3,\overline{g})$. Denote 
the projections of $\overline{T}_2$ and $\overline{T}_5$ to $Z$ with $T'_2$ and $T'_5$. Then the intersections of $T'_2$ and $T'_5$ with $\Sigma$ are transversal 
and compact, since $\Sigma$ is compact and spacelike. Consequently they are compact spacelike curves in $(Z,g')$ and the fundamental classes in $\pi_1(T'_2)$ 
resp. $\pi_1(T'_5)$ are nontrivial (The lifts to $\overline{T}_2$ and $\overline{T}_5$ cannot be closed). Therefore they intersect the projections of $\{x=2, z=z_0\}$ 
and $\{x=5, z=z_0\}$ for every $z\in \R$.

Choose a closed curve in each intersection. The fundamental classes of the projections are contained in $\pos_\Z\{-7e_2, 7e_3\}\subset \pi_1(T^3)$ on $T_2$ 
resp. in $\pos_\Z\{7e_2, 7e_3\}$ $\subset \pi_1(T_5)$ on $T_5$. Denote them by $\sigma_1\in \pos_\Z\{-7e_2, 7e_3\}$ resp. $\sigma_2 \in \pos_\Z\{7e_2, 7e_3\}$.

Since $\Sigma$ is homotopic to the spatially closed covering space, $\pi_1(\Sigma)$ can be considered as a subgroup of $\pi_1(T^3)$. But then 
$\Z\sigma_1\oplus\Z\sigma_2\subset \pi_1(\Sigma)$. 

Thus any curve representing the fundamental class $7e_3$ is of finite order in the cylindrical cover. By condition (vi) there exists a closed timelike curve $\gamma$ 
in $T^3$ with fundamental class $7e_3$. The lift $\gamma'$ of $\gamma$ to $Z$ has finite order and there exists a closed iterative of $\gamma'$. This clearly 
contradicts the causality property of $(Z,g')$.

To see why $(T^3,g)$ doesn't contain any closed transversal $1$-form, simply note that the sum of the causal future pointing closed curves 
$$\gamma_{1,2}\colon t\mapsto [(2,t,0)],[(5,-t,0)]$$ 
are nullhomologous. Therefore no closed form can be transversal to both loops.
\end{example}

For a spacetime to be of class A is purely a condition on the causal structure. So any spacetime globally conformal 
to a class A spacetime is class A as well.

Both conditions, viciousness and global hyperbolicity, on class A spacetimes are independent of each other in the sense that neither viciousness of $(M,g)$ 
implies the global hyperbolicity of $(\overline{M},\overline{g})$ (even if $\dim H_1(M,\R)>0$), 
nor does the global hyperbolicity of $(\overline{M},\overline{g})$ imply the viciousness of $(M,g)$.

Note that $\dim H_1(M,\R)>0$ for any class A spacetime. Else $\overline{M}$ would be a 
finite cover of $M$ and the causality of $(\overline M, \overline g)$ would be violated. This is due to the fact that 
any finite cover of a non-causal spacetime is again non-causal. In fact even more is true, any finite cover of a 
vicious spacetime is again vicious.

The global hyperbolicity of $(\overline M,\overline g)$ does not depend on the choice of a torsion free Abelian cover or the Abelian covering with torsion, i.e. if the 
group of deck transformations is isomorphic to $H_1(M,\Z)$ or $H_1(M,\Z)_\R\subset H_1(M,\R)$. In the subsequent discussion we will always assume that the 
group of deck transformations is isomorphic to the lattice $H_1(M,\Z)_\R$.

\begin{remark}\label{R0a}
A cover $(M',g')$ of a globally hyperbolic spacetime $(M,g)$ is always globally hyperbolic. Conversely a spacetime $(M,g)$ is globally hyperbolic if it is finitely 
covered by a globally hyperbolic spacetime.
\end{remark}

\begin{proof}
This follows directly from \cite{ha1}, proposition 1.4. 

\end{proof}

Note that the global hyperbolicity of the universal cover $(\widetilde{M},\widetilde{g})$ i.g. does not imply the global hyperbolicity of the  Abelian cover. An explicit 
example can be deduced from \cite{guediri2}.

\subsection{The stable time cone} 

Next we introduce the main technical object of these notes. Recall for $x,y\in \overline{M}$ the definition of $y-x\in H_1(M,\R)$ 
and $\rho(\gamma)$ for a Lipschitz curve $\gamma\colon [a,b]\to \overline{M}$ from section \ref{GMT}. Consider a future pointing curve $\gamma\colon [a,b]\to M$ 
parameterized by $g_R$-arclength. A sequence of such curves $\{\gamma_i\}_{i\in \mathbb N}$ is called admissible, if $L^{g_R}(\gamma_i)\to\infty$ for $i\to \infty$. 
$\T^1$ is defined to be the set of all accumulation points of sequences $\{\rho(\gamma_i)\}_{i\in\N}$ in $H_1(M,\R)$ of admissible sequences 
$\{\gamma_i\}_{i\in\N}$. $\T^1$ is compact for any compact spacetime since the stable norm of any rotation vector is bounded by $1$. Indeed by theorem \ref{3.1} 
we know that  
$$\|\rho(\gamma)\|=\frac{\|\gamma(b)-\gamma(a)\|}{L^{g_R}(\gamma)}\le 1+\frac{\std}{L^{g_R}(\gamma)}.$$
Now for any admissible sequence $\{\gamma_i\}$ we have $\|\rho(\gamma)\|\le 1+\frac{\std}{L^{g_R}(\gamma_i)}\to 1$ for $i\to \infty$. Therefore the stable norm of
all accumulation points is bounded by $1$. 

Further note that $\T^1$ is convex, if $(M,g)$ is vicious. This follows almost immediately from the following simple fact. 

\begin{note} \label{F1}
Let $M$ be compact and $(M,g)$ vicious.
Then there exists a constant $\fil(g,g_R)<\infty$ such that any two points $p,q\in M$ can be joined by a future 
pointing timelike curve with $g_R$-arclength less than $\fil(g,g_R)$.
\end{note}

Let $h,h'\in \T^1$ and $\lambda \in [0,1]$. Choose 
admissible sequences $\{\gamma_n\colon [0,T_n]\to M\}$ and $\{\eta_n\colon [0,S_n]\to M\}$ of future pointing curves with $\rho(\gamma_n)\to h$ and $\rho(\eta_n)
\to h'$. By note \ref{F1} we can assume that all $\gamma_n$ and $\eta_n$ are closed with initial point $p\in M$. Next choose sequences $\{k_n\},\{l_n\}\subseteq 
\mathbb{N}$ such that
$$\frac{k_n T_n}{k_nT_n+l_nS_n}\to \lambda$$
for $n\to\infty$. Then the sequence $\{\eta^{l_n}_n\ast \gamma^{k_n}_n\}_n$ is an admissible sequence of future pointing curves with 
$\rho(\eta^{l_n}_n\ast \gamma^{k_n}_n)\to \lambda h+(1-\lambda)h'$. This shows the convexity of $\T^1$.

We define the {\it stable time cone} $\T$ to be the cone over $\T^1$. Note that $\T$ does not depend on the choice of $g_R$, $\{k_1,\ldots ,k_b\}$ and $\omega_i\in
\alpha_i$, whereas $\T^1$ does. Reversing the time-orientation yields $-\T$ as stable time cone. $\T$ is invariant under global conformal changes of the metric 
and therefore depends only on the causal structure of $(M,g)$. It coincides with the cone of rotation vectors of structure cycles defined in appendix \ref{cones} 
(proposition \ref{T=C}). Here the cone structure is given by the future pointing tangent vectors. Further it is easy to see that this definition of 
$\T$ coincides with the ones given in section \ref{dipl} for globally conformally flat Lorentzian tori.

For compact and vicious spacetimes the stable time cone is characterized uniquely by the following property.

\begin{propositionloc}\label{P01a}
Let $(M,g)$ be a compact and vicious spacetime. Then $\T$ is the unique cone in $H_1(M,\R)$ such that there exists a constant $\err(g,g_R)<\infty$ with 
$\dist_{\|.\|}(J^+(x)-x,\T)\le \err(g,g_R)$ for all $x\in\overline{M}$, where $J^+(x)-x:=\{y-x|\;y\in J^+(x)\}$.
\end{propositionloc}
Compare this result to theorem \ref{3.1}. We will give a proof of proposition \ref{P01a} in section \ref{sec3}. 

\begin{propositionloc}\label{C16}
If $(M,g)$ is vicious, the open interior $\T^\circ$ of $\T$ is nonempty.
\end{propositionloc}
We will give a proof of this proposition in section \ref{rc}.

\subsection{Structure results}

The spacetimes $(\R\times N,-dt^2+\pi_N^\ast g_N)$ in example \ref{E0} are special cases of a more general class of spacetimes that is known in 
the literature as {\it SCTP spacetimes}. 

\begin{definitionloc}[\cite{ga84}]\label{DSCTP}
A spacetime $(M',g')$ is {\it SCTP (spatially closed time-periodic)} if 
\begin{itemize}
\item[(1)] $M'$ contains a compact spacelike Cauchy hypersurface $\Sigma$,
\item[(2)] there exists a discrete group of isometries $\psi_n\colon M'\to M'$, $n\in\Z$, such that $\Sigma_n\subseteq I^-(\Sigma_{n+1})$ and $M'=\cup_{n\in\Z}
J^+(\Sigma_n)\cap J^-(\Sigma_{n+1})$, where $\Sigma_n=\psi_n(\Sigma)$, and
\item[(3)] for each $p\in \Sigma$ there exists a positive integer $n$ such that $\psi_n(p)\in I^+(p)$.
\end{itemize}
\end{definitionloc}

It is now easy to see that every spacetime $(\R\times N,-dt^2+\pi_N^\ast g_N)$ is SCTP. First every slice $\Sigma_t:=\{t\}\times N$ is a compact Cauchy 
hypersurface. The isometries $\psi_n$ are the iterates of the isometry $\psi\colon (t,x)\to (t+1,\phi(x))$. Then $\Sigma_{t+n}$ is the image of the Cauchy 
hypersurface $\Sigma_t$ under $\psi_n$. This implies immediately $\Sigma_n\subseteq I^-(\Sigma_{n+1})$. Since $J^+(\Sigma_n)=[n,\infty)\times N$
and $J^-(\Sigma_n)=(-\infty,n]\times N$ we further obtain $\R\times N=\cup_{n\in\Z} J^+(\Sigma_n)\cap J^-(\Sigma_{n+1})$. Finally part (3) of definition \ref{DSCTP}
has already been verified for $(\R\times N,-dt^2+\pi_N^\ast g_N)$ in example \ref{E0}.

Note that the definition of SCTP spacetimes implies immediately that the group $G=\{\psi_n\}_{n\in\Z}$ operates free and properly discontinuously on $M'$, i.e. 
$M'$ induces on $M'/G$ the structure of a differentiable manifold. Indeed observe that by part (2) of definition \ref{DSCTP} every orbit $G_p$ intersects the open set 
$I^+(\Sigma_{n-1})\cap I^-(\Sigma_{n+1})$ once if $p\in \Sigma$, otherwise twice. By (2) in definition \ref{DSCTP} the family $\{I^+(\Sigma_{n-1})\cap 
I^-(\Sigma_{n+1})\}_{n\in\Z}$ forms an open cover of M'.

We denote with $\T^\ast$ the {\it dual stable time cone} of $\T$, i.e.
$$\T^\ast:=\{\alpha\in H^1(M,\R)|\;\alpha|_{\T}\ge 0\}.$$

\begin{theoremloc}\label{stab2}
Let $(M,g)$ be compact and vicious. Then the following statements are equivalent:
\begin{itemize}
\item[(i)] $(M,g)$ is of class A.
\item[(ii)] $0\notin \T^1$, especially $\T$ is a compact cone (see appendix \ref{cones}).
\item[(iii)] The open interior $(\T^\ast)^\circ$ of $\T^\ast$ is nonempty and for every $\alpha \in (\T^\ast)^\circ$ there exists a smooth $1$-form $\omega\in\alpha$ 
such that $\ker\omega_p$ is spacelike in $(TM_p,g_p)$ for all $p\in M$, i.e. $\omega$ is a closed transversal form for the cone structure of future pointing 
vectors in $(M,g)$.
\item[(iv)] $(M,g)$ admits a covering $(M',g')\to (M,g)$ by a SCTP spacetime $(M',g')$.
\end{itemize}
\end{theoremloc}

\begin{remark}
The proof of theorem \ref{stab2} will show that we can choose the covering $M'$ to be normal with $\mathcal{D}(M',M)\cong \Z$.
\end{remark}

The proof of theorem \ref{stab2} will be given in section \ref{proof}. Next we will discuss some applications of theorem \ref{stab2}. Note that the assumption of 
viciousness is essential in theorem \ref{stab2} by example \ref{E1}. 

The theorem implies immediately that the notions of SCTP spacetime and class A spacetime are equivalent in the sense that every time-orientable quotient of a 
SCTP spacetime is class A and every class A spacetime is covered by a SCTP spacetime.

\begin{corollaryloc}\label{1.10}
Let $(M,g)$ be of class A. Then there exists a constant $C_{g,g_R}<\infty$ such that 
$$L^{\overline{g}_R}(\gamma)\le C_{g,g_R}\dist(p,q)$$ 
for all $p,q\in \overline{M}$ and $\gamma$ causal connecting $p$ with $q$, where $L^{\overline{g}_R}$ denotes the length functional of the lifted Riemannian 
metric $\overline{g}_R$.
\end{corollaryloc}

\begin{proof}
By theorem \ref{stab2}(iii) the condition of class A is equivalent to the existence of a closed $1$-form $\omega$ such that $\ker\omega_p$ is spacelike for all 
$p\in M$. Since the cone of future pointing vectors is a compact cone in every tangent space, there exists a constant $c>0$ such that $\omega(v)\ge c |v|$ for all 
future pointing $v\in TM$.
Let $\tau_\omega\colon \overline{M}\to \R$ be a primitive of the pullback of $\omega$ to $\overline{M}$, i.e. $d\tau_\omega =\overline{\pi}^\ast \omega$. 
Let $p,q\in \overline{M}$ and $\gamma$ causal connecting $p$ with $q$. Then we have
$$c\cdot L^{\overline{g}_R}(\gamma)\le \tau(q)-\tau(p)\le \sup_{x\in M}\{|\omega_x|^\ast\} \dist(p,q),$$
where $|.|^\ast$ denotes the dual norm of $|.|$.

\end{proof}

\begin{corollaryloc}\label{C1}
Let $M$ be a closed manifold with $\chi(M)=0$. Then the set of class A metrics in $\Lor(M)$ is nonempty if and only 
if $M$ is diffeomorphic to a mapping torus over a closed manifold $N$. Further any class A spacetime gives rise to 
a foliation by smooth compact spacelike hypersurfaces.
\end{corollaryloc}

\begin{remark}
In the light of the differential splitting theorem for globally hyperbolic spacetimes (\cite{besa3}), the corollary is not completely surprising. In fact one could have 
expected a similar result for compact spacetimes which are covered by a globally hyperbolic one. That it fails, if one drops the assumption of viciousness, is the 
subject of example \ref{E1}.
\end{remark}

\begin{proof}[Proof of Corollary \ref{C1}]
(i) Let $(M,g)$ be of class A. Choose a cohomology class $\alpha$ with representative $\omega$ according to theorem \ref{stab2}(iii). W.l.o.g. 
we can assume that $\alpha(H_1(M,\Z)_\R)\subset \Z$. Let $\tau_\omega\colon \overline{M}\to \R$ be a primitive of $\overline{\pi}^\ast \omega$. 
By our choice of $\alpha$ every levelset $\tau_\omega^{-1}(t)$ descends to a compact hypersurface $\Sigma_t$ in $M$.

Denote with $\omega^\sharp$ the pointwise $g_R$-dual of $\omega$ and set 
$$X^\omega:=\frac{1}{g_R(\omega^\sharp,\omega^\sharp)}\omega^\sharp.$$
For the flow $\Phi^\omega$ of $X^\omega$ we know that $\Phi^\omega(.,s)\colon \Sigma_t \to \Sigma_{t+s}$ ($\omega(X^\omega)\equiv 1$). 
Choose $k\in H_1(M;\Z)_\R$ with $\alpha(k)=\min\{|\alpha(k')||\; k'\in H_1(M,\Z)_\R,\; \alpha(k')\neq 0\}$. Then $\Sigma_t=\Sigma_{t+\alpha(k)}$ 
for all $t$ and $M$ is diffeomorphic to the mapping torus $\R\times_{\Phi(.,\alpha(k))}\Sigma_t$. 

Since we have chosen $\omega$ such that $\ker \omega_p =T(\Sigma_t)_p$ is spacelike, we obtain a foliation of $M$ by by compact spacelike hypersurfaces.

(ii) Let $\R\times_\phi N$ be a mapping torus defined as the quotient of $\R\times N$ and the group of diffeomorphisms $\{(x,t)\mapsto (t+n,\phi^n(x))\}_{n\in \Z}$. 
Let $g_R$ be a Riemannian metric on $\R\times_\phi N$. We can  assume that the vector field $\partial_t'$ on $\R\times_\phi N$ induced by the embeddings 
$\R\hookrightarrow \R\times N$ is orthogonal to $N$ and of unit length. It is clear that 
$$g:=g_R-2(\partial_t')^\flat\otimes (\partial_t')^\flat$$
is a Lorentzian metric on $\R\times_\phi N$. Denote with $i_s\colon N\to \R\times_\phi N$, $x\mapsto [(s,x)]$ the natural embedding. Then by construction it is clear 
that $i_s^\ast g \equiv i_s^\ast g_R$ and $i_s(N)$ is a spacelike submanifold of $\R\times_\phi N$ for all $s\in \R$. $\partial_t'$ is timelike for $g$ and induces a 
time-orientation on $(\R\times_\phi N,g)$. It is now easy to see that the lift $g'$ of $g$ to $\R\times N$ yields a globally hyperbolic spacetime $(\R\times N, g')$. 

In order to show that $(\R\times_\phi N,g)$ is class A we have to verify that $(\R\times_\phi N,g)$ is vicious. Then theorem \ref{stab2} will imply that 
$(\R\times_\phi N,g)$ is class A. Choose $C_1<\infty$ with $i_s^\ast g_R(v,v)\le C_1 \cdot i_t^\ast g_R(v,v)$ for all $s,t\in \R$ and $v\in TN\setminus \{0\}$. 
Set $C_2:= \sup_{s}\diam (N,i_s^\ast g_R)$. Let $n> C_1\cdot C_2$ be an integer. For $(t,x)\in \R\times N$ choose a curve $\gamma_{(t,x)}\colon [0,1]\to N$ 
connecting $x$ with $\phi^n(x)$ such that $L^{i_t^\ast g_R}(\gamma_{(t,x)})\le C_2$ and parameterized w.r.t. constant $i_t^\ast g_R$-arclength. Define 
$\overline{\gamma}_{(t,x)}(\sigma):=(n\cdot \sigma +t, \gamma_{(t,x)}(\sigma))$. Then we have 
$$g'(\dot{\overline{\gamma}}_{(t,x)}(\sigma),\dot{\overline{\gamma}}_{(t,x)}(\sigma))=i^\ast_{t+n\sigma} g_R(\dot{\gamma}_{(t,x)}(\sigma),
\dot{\gamma}_{(t,x)}(\sigma))-n^2\le C_1^2 C_2^2 -n^2<0.$$
Hence $\overline{\gamma}_{(t,x)}$ is timelike. Since the endpoints of $\overline{\gamma}_{(t,x)}$ represent the same points in $\R\times_\phi N$, 
$\overline{\gamma}_{(t,x)}$ projects to a timelike loop around $[(t,x)]\in \R\times_\phi N$. This shows the viciousness of $(\R\times_\phi N,g)$ and finishes the 
proof.

\end{proof}

\begin{corollaryloc}\label{stab}
For every compact manifold $M$ the set 
$$\{g\in \Lor(M)|\; (M,g)\text{ is of class A}\}$$ 
is open in the $C^0$-topology on $\Lor(M)$.
\end{corollaryloc}

\begin{proof}[Proof of Corollary \ref{stab}]
The openness of the {\it viciousness} condition was already proven in note \ref{F1}. Consequently it remains to verify the condition {\it $(\overline M,\overline g)$ 
globally hyperbolic} is open in the $C^0$ topology on $\Lor(M)$ in the case that $(M,g)$ is vicious. 

Consider a smooth and closed $1$-form $\omega$ on $M$ such that $\ker\omega_p$ is spacelike for all $p\in M$. Next consider the set $\mathcal{G}(\omega)
\subset Lor(M)$ of metrics $g_1$ such that $\ker\omega_p$ is $g_1$-spacelike for all $p\in M$. $\mathcal{G}(\omega)$ is certainly an open neighborhood of $g$ in 
$Lor(M)$. 

Let $g_1\in \mathcal{G}(\omega)$. We want to show that the lift $\overline{g}_1$ of $g_1$ to $\overline{M}$ is globally hyperbolic. Since $\ker\omega_p$ is 
$g_1$-spacelike for all $p\in M$, any primitive $\tau_\omega\colon \overline{M}\to \R$ of $\overline{\pi}^\ast \omega$ is a temporal function for $(\overline{M},
\overline{g}_1)$. By the compactness of $M$ there exists $\e_1>0$ such that we have $|d\tau_\omega (v)|\ge \e_1 |v|$ for all $\overline{g}_1$-nonspacelike $v\in 
T\overline{M}$. 

Let $\gamma\colon \R\to \overline{M}$ be an inextendible $\overline{g}_1$-nonspacelike curve parameterized w.r.t. $\overline{g}_R$-arclength. W.l.o.g. we can 
assume that $\tau_\omega\circ \gamma$ is increasing, i.e. $d\tau_\omega(\dot{\gamma}(t))\ge\e_1 |\dot{\gamma}(t)|$ whenever $\dot{\gamma}(t)$ exists. Let 
$\Sigma:=\tau_\omega^{-1}(\sigma)$ be any level set of $\tau_\omega$. We want to show that $\gamma$ intersects $\Sigma$ exactly once. Then we are done, 
since by that property $\Sigma$ is a Cauchy hypersurface of $(\overline{M},\overline{g}_1)$. This is well known to be equivalent to the global hyperbolicity of 
$(\overline{M},\overline{g}_1)$.

Set $\sigma_0:=\tau_\omega(\gamma(0))$. For $r\ge \frac{|\sigma-\sigma_0|}{\e_1}$ we have 
$$|\tau_\omega(\gamma(\pm r))-\sigma_0|=|\int_0^{\pm r}d\tau_\omega (\dot{\gamma})|\ge \e_1 r\ge |\sigma-\sigma_0|.$$
Then $t$ is either contained in the interval $[\tau_\omega(\gamma(-r)),\sigma_0]$ or $[\sigma_0,\tau_\omega(\gamma(r))]$. By the intermediate value theorem 
$\gamma$ has to intersect $\Sigma$. Since $\tau_\omega$ is strictly increasing along $\gamma$, the intersection is unique.

\end{proof}

Recall the definition of standard SCTP spacetimes from the introduction. A spacetimes $(\R\times \Sigma, -f^2dt^2+ g_\Sigma)$, where $f$ is a positive function on 
$\R\times \Sigma$ and $g_\Sigma$ is a Riemannian metric on $\Sigma$ i.g. depending on the $t$-coordinate as well, is called standard SCTP spacetime if there 
exists an isometry $\psi \colon \R\times\Sigma\to \R\times \Sigma$ of the form $\psi(t,x)=(t+1,\phi(x))$ or $\psi(t,x)=(-t+1,\phi(x))$ for some diffeomorphism 
$\phi \colon \Sigma \to \Sigma$. 

With a slight modification of the constants in the proof of corollary \ref{C1} one sees that every standard SCTP spacetime is a special case of a SCTP spacetime.

\begin{corollaryloc}\label{C2}
Let $(M',g')$ be a SCTP spacetime with a smooth Cauchy hypersurface $\Sigma$. Then there exists an isometry $\Phi\colon (M',g') \cong (\R\times \Sigma,
-f^2dt^2 +g_\Sigma)$ to a standard SCTP spacetime such that $\Phi\circ \psi_n\circ \Phi^{-1} = (t,x)\mapsto  (t+n,\phi^n(x))$ if the group $\{\psi_n\}_{n\in\Z}$ is time 
orientation preserving and $(t,x)\mapsto ((-1)^n t+n, \phi^n(x))$ if $\{\psi_n\}_{n\in\Z}$ is time orientation reversing.
\end{corollaryloc}

\begin{remark}
The assumption of smoothness for the Cauchy hypersurface $\Sigma$ poses no restriction on the spacetimes considered. In fact one can easily show that any 
SCTP spacetime admits a smooth Cauchy hypersurface meeting the conditions in definition \ref{DSCTP}.
\end{remark}

\begin{proof}
(i) We will treat the case that $G:=\{\psi_n\}_{n\in\Z}$ is time orientation preserving first. Since $(M',g')$ is SCTP, the quotient $M:=M'/G$ together with the induced 
time orientable Lorentzian metric $g$ is class A according to theorem \ref{stab2}. 

By the definition of SCTP spacetimes the natural projection $\pi'\colon M'\to M$ restricted to $\Sigma$ is injective and the image of the map $(\pi'|_\Sigma)_\ast
\colon H_1(\Sigma,\R)\to H_1(M,\R)$ is a hyperplane in $H_1(M,\R)$. Indeed let $\gamma\colon S^1\to M$ be any loop. Recall that $M'$ is diffeomorphic to 
$\R\times \Sigma$ by the  Lorentzian splitting theorem. If the lift $\gamma'$ of $\gamma$ to $M'$ is closed, $\gamma'$ defines a homology class in $H_1(\Sigma,
\Z)$. If $\gamma'$ is not closed, it connects a point $x\in M'$ with $\psi_n(x)$ for some $n\in \Z$. By the same argument as before we see that for any other curve 
$\eta'$ between $x$ and $\psi_n(x)$ the loop $(\eta')^{-1}\ast \gamma'$ defines a homology class in $H_1(\Sigma,\Z)$. This shows that the image of 
$H_1(\Sigma,\R)$ under $(\pi'|_{\Sigma})_\ast$ is a rational hyperplane in $H_1(M,\R)$. 

Choose  $\alpha\in H^1(M,\R)$ with $\ker\alpha =H$. We claim that $\alpha$ or $-\alpha\in (\T^\ast)^\circ$. This will follow from proposition \ref{P01a}. 

It is clear that $\alpha$ or $-\alpha \in \T^\ast$ since else there exist timelike loops $\eta_1,\eta_2\subseteq M$ with homology classes $h_1,h_2\in \T$ satisfying
$\langle \alpha ,h_1\rangle >0$ and $\langle \alpha ,h_2\rangle <0$. From this one can easily construct futurepointing curves in $M'$ intersecting $\Sigma$ 
arbitrarily often, thus contradicting the Cauchy hypersurface property of $\Sigma$.  Choose $\alpha$ such that $\alpha|_{\T}\ge 0$.

Assume now that $\alpha \notin (\T^\ast)^\circ$. Then, according to proposition \ref{P01a}, there exist a sequence of admissible future pointing curves 
$\gamma_n\colon [-a_n,a_n]\to M$ such that 
$$\dist\nolimits_{\|.\|}(\gamma_n(a_n)-\gamma_n(-a_n),\ker\alpha) \le \err(g,g_R).$$ 
Now the limit curve lemma (\cite{bee} lemma 14.2) implies that there exists a subsequence that converges uniformly on compact subsets to an inextendable future 
pointing curve $\gamma\colon \R\to M$. Lift $\gamma$ to a future pointing curve $\gamma'\colon \R\to M'$. Note that since $(M,g)$ is vicious the homology class 
$\gamma(t)-\gamma(s)\in H_1(M,\R)$ $(s<t)$ stay within a bounded distance from $\ker\alpha$. 

Fix a curve $\eta'$ in $M'$ connected a point $x\in M'$ with $\psi_1(x)$ and denote with $h_1$ the homology class of the projection to $M$. 
By construction we have $\langle \alpha ,h_1\rangle  > 0$. This shows that $\gamma$ remains in a compact subset of 
$M'$ and thus contradicting the global hyperbolicity of $(M',g')$. This shows that $\alpha\in (\T^\ast)^\circ$.

Choose $\omega\in \alpha$ according to theorem \ref{stab2}, i.e. $\ker \omega_p$ is spacelike for all $p\in M$. Note that the pullback $\omega'$ of $\omega$
to $M'$ is exact, since it vanishes on $H_1(M',\R)\cong H_1(\Sigma,\R)$. Thus there exists a primitive $\tau\in C^\infty (M')$ of $\omega'$. $\tau$ is a Cauchy 
temporal function equivariant under the group of deck transformations $\mathcal{D}(M',M)\cong G$ by construction. By rescaling $\omega$ (and with it $\tau$) we 
can assume that $\tau(\psi_n(x))=\tau(x)+n$ for all $n\in \Z$. 

Since $\tau$ has no critical points, $(M',g')$ is isometric to the product $(\R\times \tau^{-1}(0),-fdt^2 +g_{\tau^{-1}(0)})$ where $f\in C^\infty(\R\times \Sigma)$ and
$g_{\tau^{-1}(0)}$ is a Riemannian metric depending i.g. on the $t$-coordinate.

Since $\Sigma$ as well as $\tau^{-1}(0)$ are Cauchy hypersurfaces in $(M',g')$, $\Sigma$ and $\tau^{-1}(0)$ are diffeomorphic (e.g. the image of $\Sigma$ in 
$\R\times \tau^{-1}(0)$ is the graph of a smooth function over $\tau^{-1}(0)$). Consequently $(M',g')$ is isometric via an isometry $\Phi$ to the spacetime 
$(\R\times \Sigma,-fdt^2 +g_{\Sigma})$ such that $\Phi\circ \psi_n\circ \Phi^{-1} = (t,x)\mapsto  (t+n,\phi^n(x))$ for all $n\in\Z$. 

(ii) If $G$ does not preserve the time orientation consider the subgroup $G^e :=\{\psi_{2n}\}_{n\in\Z}$. It is clear that $G^e$ preserves the time orientation. 
Now we apply the construction from above to the spacetime $(M'/G^e, g^e)$ where $g^e$ denotes the induced Lorentzian metric. The resulting $1$-form 
$\omega'$ is invariant under $G^e$. Choose a generator $\psi$ of $G$ and consider the exact $1$-form $\omega'':=\omega'-\psi^\ast \omega'$. 
$\omega''$ is invariant under $G$ and $\ker \omega''_{p'}$ is spacelike for all $p'\in M'$ since $\psi$ reverses the time orientation.
Now the result follows in the same way as for the time orientation preserving case except that $\Phi\circ \psi_n\circ \Phi^{-1} = (t,x)\mapsto  ((-1)^n t+n,\phi^n(x))$

\end{proof}

\subsection{The coarse-Lipschitz property}

When comparing Lorentzian geometry with Riemannian geometry the question of Lipschitz continuity of the time separation appears naturally. As Minkowski space 
shows this question has no general positive answer for neither the entire set $J:=\{(p,q)|\; q\in J^\pm (p)\}$ nor $I:=\{(p,q)|\; q\in I^\pm (p)\}$. It received some 
attention in the literature, though, in connection with the Cheeger-Gromoll splitting theorem for Lorentzian manifolds (see \cite{es}).

Recall the definition of $\T_\e:=\{h\in\T|\; \dist_{\|.\|}(h,\partial \T)\ge \e \|h\|\}$ for the cone $\mathfrak{K}=\T\subseteq H_1(M,\mathbb{R})$ and $\e >0$
from corollary \ref{C100}.

\begin{theoremloc}\label{T16}
Let $(M,g)$ be of class A. Then for every $\e >0$ there exists $L_c(\e)<\infty$, such that 
  $$|d(x,y)-d(z,w)|\le L_c(\e)(\dist(x,z)+\dist(y,w)+1)$$
for all $(x,y),(z,w)\in \overline{M}\times\overline{M}$ with $y-x,w-z\in \T_\e$.
\end{theoremloc}

The stronger question of Lipschitz continuity is unanswered at this point in this generality. Note that the assumptions 
of theorem \ref{T16} are not empty due to proposition \ref{C16}. 

The proof of theorem \ref{T16} consists of showing that future pointing curves $\gamma$ from $x$ to $y$ can be used to ``build''
future pointing curves from $z$ to $w$, with the additional property that the length of the part of $\gamma$, 
which has to be sacrificed in the construction, is congruent to $\dist(x,z)+\dist(y,w)+1$.
The arguments in the proof are similar to the socalled cut-and-paste arguments employed in \cite{ba}, \cite{ma} et al..

\section{Proof of Theorem \ref{stab2}}\label{proof}

The proof of theorem \ref{stab2} will be divided into several steps. The first steps will prove the implications (ii)$\Rightarrow$ (iii)$\Rightarrow$(i). In section 
\ref{sec2a} we will prove the implications (iii)$\Rightarrow$(iv) and (iv)$\Rightarrow$(ii). The implication (i)$\Rightarrow$(ii) is the subject of subsection \ref{sec3}.

Recall theorem \ref{stab2}:
\begin{theoremlocstar}
Let $(M,g)$ be compact and vicious. Then the following statements are equivalent:
\begin{itemize}
\item[(i)] $(M,g)$ is of class A.
\item[(ii)] $0\notin \T^1$, especially $\T$ is a compact cone.
\item[(iii)] $(\T^\ast)^\circ\neq\emptyset$ and for every $\alpha \in (\T^\ast)^\circ$ 
there exists a smooth $1$-form $\omega\in\alpha$ such that $\ker\omega_p$ is a spacelike in $(TM_p,g_p)$ for all 
$p\in M$, i.e. $\omega$ is a closed transversal form for the cone structure of future pointing 
vectors in $(M,g)$.
\item[(iv)] $(M,g)$ admits a covering $(M',g')\to (M,g)$ by a SCTP spacetime $(M',g')$.
\end{itemize}
\end{theoremlocstar}

\subsection{$(ii)\Rightarrow(iii)\Rightarrow(i)$}

\begin{lemmaloc}
Let $(M,g)$ be a compact and vicious spacetime. Denote with $C_g$ the cone structure of future pointing vectors in $TM$. If $0\notin \T^1$ then no non-trivial 
structure cycle is homologous to zero.
\end{lemmaloc}

\begin{proof}
Let $c$ be a non-trivial structure cycle. From lemma \ref{L101} we know that $c$ can be approximated by foliation cycles $\hat{c}$ of $1$-dimensional timelike 
foliations $\mathfrak{F}$. Let $X$ be a future pointing timelike vector field tangent to $\mathfrak{F}$ with $|X|\equiv 1$. We can choose a finite Borel measure 
$\mu$ invariant under the flow $\phi_t$ of $X$ such that $\hat{c}=\int X d\mu$ according to proposition \ref{II.24}. By the theorem of Krein-Milman $\mu$ is 
approximated in the weak-$\ast$ topology on currents by finite positive combinations of $\phi_t$-ergodic probability measures.

Let $\sum_i \lambda_i \mu_i$ be a finite positive combination of $\phi_t$-ergodic probability measures $\mu_i$. Then the current $c^X:=\sum \lambda_i\int X 
d\mu_i$ $(\lambda_i >0)$ is a foliation cycle of $\mathfrak{F}$. Since $0\notin \T^1$ there exists a cohomology class $\alpha$ such that $\|h\|\ge \alpha(h)>0$ for 
all $h\in \T^1$. From the ergodic theorem follows that $h_i\in\T^1$ for the homology class $h_i$ of the structure cycle $\int X d\mu_i$. Therefore we have 
$\|h^X\|\ge \sum_i \lambda_i$ for the homology class $h^X$ of $c^X$.

If $c$ is homologous to zero, the homology classes $h^X$ approximate the zero homology class. But then the sum $\sum_i \lambda_i$ has to approximate
$0$, thus showing that $c$ is the trivial structure current.
\end{proof}

\begin{proof}[Proof of (ii) $\Rightarrow$ (iii)]
By elementary convex geometry we see that $(\T^\ast)^\circ\neq \emptyset$. The rest is a consequence of theorem \ref{TS10} (ii) and (iv). More precisely, since 
$(M,C_g)$ contains no non trivial structure cycles homologous to zero there is a closed transversal $1$-form. Since $(M,g)$ contains closed causal curves there 
are structure cycles of $(M,C_g)$. Thus by (iv) of theorem \ref{TS10} the interior of $\T^\ast$ consists of classes of closed transversal $1$-forms.
\end{proof}

\begin{proof}[Proof of (iii) $\Rightarrow$ (i)] 
Let $\omega \in \alpha \in (\T^\ast)^\circ$ be a smooth $1$-form such that $\ker \omega_p$ is spacelike for all $p\in M$. Then every primitive $\tau_\omega\colon
\overline{M}\to \R$ of $\overline{\pi}^\ast \omega$ is, by construction, a temporal function. The global hyperbolicity of $(\overline{M},\overline{g})$ follows, if 
$\tau_\omega$ is a Cauchy temporal function, i.e. $\tau_\omega\circ \gamma$ is onto $\R$ for every inextendible causal curve $\gamma\colon (a,b)\to 
\overline{M}$. Reparameterization of $\gamma$ does not alter the claim. Therefore we can assume that $\gamma$ is Lipschitz. 

Choose $c>0$ such that $d\tau_\omega(v)\ge c|v|$ for all future pointing $v\in T\overline{M}$. Then we have 
$$\tau_\omega(\gamma(t))-\tau_\omega(\gamma(s))\ge c L^{\overline{g}_R}(\gamma|_{[s,t]})$$
for all $s,t\in (a,b)$. Since $\overline{g}_R$ is complete $L^{\overline{g}_R}(\gamma|_{[s,t]})\to \infty$ for $s\to a$ or $t\to b$. This shows the claim and the 
global hyperbolicity of $(\overline{M},\overline{g})$.
\end{proof}

\subsection{$(iii)\Rightarrow (iv)$, $(iv)\Rightarrow (ii)$}\label{sec2a}

\begin{proof}[Proof of (iii) $\Rightarrow$ (iv)] 
Let $\alpha\in (\T^\ast)^\circ$ such that $\alpha(H_1(M,\Z)_\R\subset \Q$. Then there exist linearly independent $k_1,\ldots ,k_{b-1}\in \ker\alpha\cap 
H_1(M,\mathbb Z)_\R$. Define $M':=\overline M/<k_1,\ldots ,k_{b-1}>_\mathbb Z$. Choose $\omega\in \alpha$ according to (iii). The pullback $\omega'$ of 
$\omega$ to $M'$ is exact by construction and every primitive $\tau$ of the $\omega'$ is a Cauchy temporal function on $(M',g')$ where $g'$ denotes the induced 
Lorentzian metric. Therefore we have that $(M',g')$ is globally hyperbolic and the spacelike Cauchy hypersurfaces $\{\tau\equiv t\}$ are compact. 

Choose any $\psi\in \mathcal{D}(M',M)\cong H_1(M,\Z)_\R/<k_1,\ldots ,k_{b-1}>_\Z$ such that $\tau\circ \psi <\tau$. Then $\{\psi^n\}_{n\in\Z}$ is a discrete group 
of isometries as in definition \ref{DSCTP} for $\Sigma_n:=\{\tau\circ\psi^n\equiv 0\}$. 

Part (3) of definition \ref{DSCTP} follows from the assumption of viciousness on the spacetime $(M,g)$. 
\end{proof}

\begin{remark}
The construction shows that the covering $M'\to M$ is normal and $ \mathcal{D}(M',M)\cong H_1(M,\Z)_\R/<k_1,\ldots ,k_{b-1}>_\Z\cong\Z$.
\end{remark}

\begin{proof}[Proof of (iv) $\Rightarrow$ (ii)]
Let $(M',g')$ be a SCTP covering spacetime of a compact and vicious spacetime $(M,g)$. We have to show that there exists a $c>0$ such that $\|\rho(\gamma_n)\|
\ge c$ for any admissible sequence of future pointing curves $\gamma_n\colon [a_n,b_n]\to M$. 

Let $\Sigma$ be a compact Cauchy hypersurface in $(M',g')$ and $\psi_n\colon M'\to M'$ isometries as in definition \ref{DSCTP}. Denote with $g'_R$ the lift of 
$g_R$ to $M'$. We can choose $c'>0$ such that the $g'_R$-arclength of any causal curve connecting the Cauchy hypersurfaces $\Sigma_m$ and 
$\Sigma_{m+1}$ is bounded from above by $c'$. This follows from the compactness of $J^+(\Sigma_m)\cap J^-(\Sigma_{m+1})$ and 
the fact that the $\Sigma_{m}$'s are Cauchy hypersurfaces. Denote the lift of $\gamma_n$ to $M'$ with $\gamma'_n$. Then $\gamma'_n(a_n)\in J^+(\Sigma_m)
\cap J^-(\Sigma_{m+1})$ and $\gamma'_n(b_n)\in J^+(\Sigma_{m+k})\cap J^-(\Sigma_{m+k+1})$ imply that
\begin{equation}\label{E120}
L^{g_R}(\gamma_n)\le (k+2) \cdot c'.
\end{equation}

Note that there exists $m_0>0$ such that any curve $\eta'$ connecting $x\in M'$ with $\psi_{m_0}(x)$ projects to a closed curve $\eta$ in $M$. Since $\psi_m$ is 
fix-point free for all $m$, the homology class of $\eta$ is nontrivial. Denote with $c''>0$ the minimum of the stable norm of these homology classes. Then if 
$\gamma'_n(a_n)\in J^+(\Sigma_m)\cap J^-(\Sigma_{m+1})$ and $\gamma'_n(b_n)\in J^+(\Sigma_{m+k})\cap J^-(\Sigma_{m+k+1})$ we have 
\begin{equation}\label{E121}
\|\gamma_n(b_n)-\gamma_n(a_n)\|\ge \frac{k}{m_0}c''-\std-2\diam(J^+(\Sigma)\cap J^-(\Sigma_{m_0})).
\end{equation}
Now for any sequence of curves $\gamma_n\colon[a_n,b_n]\to M$ define the sequence $k_n\in \Z$ by considering a lift $\gamma'_n$ of $\gamma_n$ 
to $M'$. Then there exists $m,k_n\in \Z$ such that $\gamma'_n(a_n)\in J^+(\Sigma_m)\cap J^-(\Sigma_{m+1})$ and $\gamma'_n(b_n)\in J^+(\Sigma_{m+k_n})
\cap J^-(\Sigma_{m+k_n+1})$. If the sequence $\gamma_n$ is an admissible sequence of future pointing curves, \eqref{E120} shows that $k_n\to \infty$.
Now combining \eqref{E120} and \eqref{E121} yields the claim.
\end{proof}

\subsection{$(i)\Rightarrow(ii)$}\label{sec3}

In order to prove the implication (i)$\Rightarrow$(ii) in theorem \ref{stab2}, we use proposition \ref{P01a}. 
The proof of proposition \ref{P01a} consists of a modification of a method introduced by 
D. Yu Burago in \cite{bu}.

\begin{definitionloc}\label{D2}
Let $(M,g)$ be compact and vicious. For $h \in H_1(M,\Z)_\R$ and $x\in \overline M$ define
$$\mathfrak{f}_x(h):=\min \{\dist(x+h,z)|\; z\in J^+(x)\}\text{ and } 
\mathfrak{f}(h):=\min \{\mathfrak{f}_x(h)|\; x\in \overline M\}.$$
\end{definitionloc}

Note that $x\mapsto \mathfrak{f}_x(h)$ is invariant under the action of $\mathcal{D}(\overline{M},M)$ for all $h\in H_1(M,\Z)_\R$. Consequently $\mathfrak{f}$ is 
well defined. 

Recall the statement of proposition \ref{P01a}. 

\begin{propositionlocstar}
Let $(M,g)$ be a compact and vicious spacetime. Then $\T$ is the unique cone in $H_1(M,\R)$ such that there exists 
a constant $\err(g,g_R)<\infty$ with $\dist_{\|.\|}(J^+(x)-x,\T)\le\err(g,g_R)$ for all $x\in\overline{M}$, where 
$J^+(x)-x:=\{y-x|\;y\in J^+(x)\}$.
\end{propositionlocstar}

It is easy to see that there exists $K<\infty$ such that $J^+(x)-x\subset B^{\|.\|}_K(\T)$ for all $x\in\overline{M}$. In fact we know that 
$$\dist\nolimits_{\|.\|}(\gamma(b)-\gamma(a),\T)\le \fil(g,g_R)+\std$$
by  note \ref{F1} and theorem \ref{3.1} for any future pointing curve $\gamma\colon [a,b]\to M$. Therefore $J^+(x)-x$ is contained in the 
$\fil(g,g_R)+\std$-neighborhood of $\T$ for every $x\in\overline{M}$. 

It remains to show the existence of a real number $K<\infty$ such that $\T$ is contained in the $K$-neighborhood of $J^+(x)-x$.
This is far more involved. First we prove that $\mathfrak{f}$ has the coarse-Lipschitz property. 

\begin{lemmaloc}
There exists $C<\infty$ such that 
$$|\mathfrak{f}(h_1)-\mathfrak{f}(h_2)|\le \|h_1-h_2\|+C$$
for all $h_1,h_2\in H_1(M,\Z)_\R$.
\end{lemmaloc}

\begin{proof}
Let $h_1,h_2\in H_1(M,\Z)_\R$. Choose $x,y\in \overline M$ with $\mathfrak{f}(h_1)=\mathfrak{f}_x(h_1)$, $\mathfrak{f}(h_2)=\mathfrak{f}_y(h_2)$ and $\dist(x,y)\le 
\diam (M,g_R)$. Since $\mathfrak{f}_x(h_2)\le \mathfrak{f}_y(h_2)+\diam (M,g_R)$ we have
$$|\mathfrak{f}_x(h_2)-\mathfrak{f}_y(h_2)|\le \diam (M,g_R),$$
Further we have $\mathfrak{f}_x(h_1)\le \mathfrak{f}_x(h_2)+\dist(x+h_1,x+h_2)$ where $x+h:=\{z|\; z-x=h\}$. An immediate consequence of theorem \ref{3.1} is
$$|\dist(x+h_1,x+h_2)-\|h_1-h_2\||\le D'$$
for some constant $D'<\infty$. Now we get 
\begin{align*}
|\mathfrak{f}(h_1)-\mathfrak{f}(h_2)|&\le |\mathfrak{f}_x(h_1)-\mathfrak{f}_x(h_2)|+|\mathfrak{f}_x(h_2)
-\mathfrak{f}_y(h_2)|\\
&\le \dist(x+h_1,x+h_2)+\diam(M,g_R)\\
&\le \|h_1-h_2\|+D'+\diam(M,g_R).
\end{align*}
\end{proof}

The following lemma differs slightly from the statement of lemma 1 in \cite{bu}. We leave the proof to the reader since it is an almost literally transcription of the 
proof given in therein.
\begin{lemmaloc}\label{lemmabu}
Let $C<\infty$ and $F\colon \N\rightarrow [0,\infty)$ be a coarse-Lipschitz function with 
\begin{enumerate}
\item $2F(s)-F(2s)\le C$ and
\item $F(\kappa s)-\kappa F(s)\le C$ for $\kappa=2,3$
\end{enumerate}
and all $s\in\N$. Then there exists an $\mathfrak{a}\in \R$ such that $|F(s)-\mathfrak{a}s|\le 2C$ for all $s\in\N$.
\end{lemmaloc}

Now we want to apply this lemma to $\mathfrak{f}$. First we fix a trivial fact.
\begin{note}
Consider $\mathfrak{f}$ as in definition \ref{D2}. Then we have
$\mathfrak{f}(2h)\le 2\mathfrak{f}(h)$ and $\mathfrak{f}(3h)\le 3\mathfrak{f}(h)$ for all $h\in H_1(M,\Z)_\R$.
\end{note}
The next lemma requires more attention.
\begin{lemmaloc}\label{L14}
Consider $\mathfrak{f}$ as in definition \ref{D2}. Then there exists a constant $C=C(g,g_R)<\infty$ such that
$\mathfrak{f}(2h)\ge 2\mathfrak{f}(h)-C$ for all $h\in H_1(M,\Z)_\R$.
\end{lemmaloc}

We will need the following lemma contained in \cite{bu}. 

\begin{lemmaloc}\label{lembu2}
Let $V$ be a real vector space of dimension $b<\infty$ and $\gamma\colon [a,b]\to V$ a continuous curve. 
Then there exist no more than $[b/2]$-many essentially disjoint subintervals $[a_i,b_i]\subset [a,b]$ $(1\le i\le k\le [b/2])$
such that 
$$\sum_{i=1}^k [\gamma(b_i)-\gamma(a_i)] =\frac{1}{2}[\gamma(b)-\gamma(a)].$$
\end{lemmaloc}
The proof is a nontrivial application of the theorem of Borsuk-Ulam and can be found in \cite{bu}. 

\begin{proof}[Proof of Lemma \ref{L14}]
We have already seen above that 
$$|\mathfrak{f}_x(h)-\mathfrak{f}_y(h)|\le \diam(M,g_R)$$ 
for all $x,y\in\overline{M}$ and $h\in H_1(M,\Z)_\R$. Let $h\in H_1(M,\Z)_\R$ be given. Fix $x\in\overline{M}$. Further choose a future pointing curve $\gamma
\colon [0,T]\to\overline{M}$ with $\gamma(0)=x$ and $\dist(\gamma(T),x+2h)=\mathfrak{f}_x(2h)$. Now consider the curve $\gamma_D\colon [0,T]\to H_1(M,\R)$, 
$t\mapsto \gamma(t)-\gamma(0)$. The pair $(H_1(M,\R),\gamma_D)$ obviously meets the assumptions of lemma \ref{lembu2}. Consequently there exist at most 
$[b/2]$-many intervals $[s_i,t_i]\subset [0,T]$ $(1\le i\le k\le [b/2])$ with 
$$\sum [\gamma_D(t_i)-\gamma_D(s_i)]=\frac{1}{2}[\gamma_D(T)-\gamma_D(0)].$$
W.l.o.g. we can assume that $a_1=0$. In the other case simply consider the complementary intervals $[t_{i-1},s_i]$. Note that 
$$\|\sum [\gamma_D(t_i)-\gamma_D(s_i)]-h\|\le \frac{1}{2}(\std +\mathfrak{f}_x(2h)),$$ 
since $\|[\gamma(T)-\gamma(0)]-2h\|\le \std +\mathfrak{f}_x(2h)$. Choose inductively deck transformations $k_i$ starting with $k_1:=0\in H_1(M,\Z)_\R$ and 
$k_i\in H_1(M,\Z)_\R$ for $i\ge 2$ such that $\gamma(s_i)+k_i\in J^+(\gamma(t_{i-1})+k_{i-1})$ and $\dist(\gamma(t_{i-1})+k_{i-1},\gamma(s_i)+k_i)\le \fil(g,g_R)$. 
Join $\gamma(t_{i-1})+k_{i-1}$ and $\gamma(s_{i})+k_i$ by a future pointing curve with length at most $\fil(g,g_R)$. The resulting future pointing curve $\zeta\colon 
[0,T']\to \overline{M}$ then satisfies 
$$\|\zeta(T')-\zeta(0)-h\|\le [b/2]\fil(g,g_R)+\frac{1}{2}(\std+\mathfrak{f}_x(2h)).$$
Since by theorem \ref{3.1} we have $\dist(\zeta(T'),x+h)\le \|\zeta(T')-\zeta(0)-h\|+\std$, the lemma follows for $C:=2[b/2]\fil(g,g_R)+3\std$.
\end{proof}

Now we can apply lemma \ref{lemmabu} to the function $n\mapsto \mathfrak{f}(nh)$ for every $h\in H_1(M,\Z)_\R$. As a result we get $\mathfrak{a}(h)\in \R$ with 
$|\mathfrak{a}(h)n-\mathfrak{f}(nh)|\le 2C$ for all $n\in\N$. This immediately implies positive homogeneity of $\mathfrak{a}$. Combining this we get the following 
fact.
\begin{note}\label{100}
There exists a map $\mathfrak{a}\colon H_1(M,\Z)_\R\to \R$ and $C<\infty$ such that
\begin{enumerate}
\item $\mathfrak{a}$ is positively homogenous of degree one, i.e. $\mathfrak{a}(n h)=n \mathfrak{a}(h)$ 
for all $n\in\N$ and 
\item $|\mathfrak{f}(h)-\mathfrak{a}(h)|\le 2C$
\end{enumerate}
for every $h\in H_1(M,\Z)_\R$.
\end{note}

\begin{note}\label{101}
We have $\mathfrak{a}(h)=\dist_{\|.\|}(h,\T)$ for all $h\in H_1(M,\Z)_\R$. 
\end{note}

\begin{proof}
Let $h\in H_1(M,\Z)_\R$. For $n\in\N$ let $\gamma_n\colon [0,T]\rightarrow \overline{M}$ be a future pointing curve with 
$$\dist(\gamma_n(0)+nh,\gamma_n(T))=\mathfrak{f}(nh).$$
Then with theorem \ref{3.1} and note \ref{100} we get
\begin{align*}
|\|n h-(\gamma_n(T)-\gamma_n(0))\|&-\mathfrak{a}(h)n|\\
   &\le |\dist(\gamma_n(0)+nh,\gamma_n(T))-\mathfrak{a}(h)n|+D\\
                                        &\le 2C+D.
\end{align*}
Now we have 
$$\lim_{n\rightarrow \infty}\frac{1}{n}\|nh-(\gamma_n(T)-\gamma_n(0))\|=\dist\nolimits_{\|.\|}(h,\T)$$ 
since otherwise the distance between $\gamma_n(0)+nh$ and $\gamma_n(T)$ would not be minimal.
\end{proof}

To prove the remaining inclusion in the proof of proposition \ref{P01a} observe that by note \ref{100}, \ref{101} and the fact 
that $H_1(M,\Z)_\R$ is a cocompact lattice in $H_1(M,\R)$, the Hausdorff distance between 
$\T= \dist_{\|.\|}(.,\T)^{-1}(0)$ and 
$$\mathfrak{f}^{-1}(0)=\{h\in H_1(M,\Z)_\R|\; \exists x\in \overline{M}\text{ with }x+h\in J^+(x)\}$$ 
is bounded by $2C$. Further observe that by note \ref{F1} there exists a constant $C'=C'(g,g_R)<\infty$ such that 
$$\dist\nolimits_{\|.\|}(J^+(x)-x,J^+(y)-y)\le C'$$
for all $x,y\in \overline{M}$.
Thus the Hausdorff distance of $\mathfrak{f}^{-1}(0)$ and $J^+(x)-x$ is uniformly bounded in $x$. 
Now combining these arguments we get the claim of proposition \ref{P01a}.

\begin{proof}[Proof of (i) $\Rightarrow$ (ii)]
The first step is to confirm that $\T$ does not contain a nontrivial linear subspace. This is done by contradiction.

Assume $\T$ contains a linear subspace $V\neq \{0\}$. Choose $h\in V\setminus\{0\}$. By proposition \ref{P01a} 
there exists for any $h'\in V$ a homology class $h'_x\in J^+(x)-x$ with $\|h'-h'_x\|\le \err(g,g_R)$ for 
any $x\in \overline{M}$. We can choose future pointing curves $\gamma^+,\gamma^-\colon [0,1]\to\overline{M}$ 
with 
\begin{align*}
\|\gamma^+(1)-\gamma^+(0)-h\|&,\; \|\gamma^-(1)-\gamma^-(0)+h\|\le \err(g,g_R),\\ 
\dist(\gamma^+(1),\gamma^-(0))\le &\fil(g,g_R)\text{ and }\gamma^-(0)\in J^+(\gamma^+(1)).
\end{align*}
Then $\dist(\gamma^+(0),\gamma^-(1))\le 2C+\fil(g,g_R)+\std$ and we can construct a future pointing curve 
$\zeta_h$ connecting $\gamma^+(0)$ with $\gamma^-(1)$ of $g_R$-length at least $2\|h\|-\std$. Choose a sequence of 
future pointing curves $\zeta_n:=\zeta_{h_n}\colon [0,T_n]\to \overline{M}$ for an unbounded sequence $h_n\in V$. 
By passing to a subsequence we can assume $\zeta_n(0)\to p'$ and $\zeta_n(1)\to q'$. Choose any point 
$p\in I^-(p')$ and $q\in I^+(q')$. Then $J^+(p)\cap J^-(q)$ is not compact, 
thus contradicting the global hyperbolicity of $(\overline{M},\overline{g})$. Consequently $\T$ cannot contain any 
nontrivial linear subspaces.

If $\T$ doesn't contain a nontrivial linear subspace we can choose a cohomology class $\alpha$ with 
$\ker\alpha\cap\T=\{0\}$. Consequently there exists $\e>0$ such that $\alpha(h)\ge \e\|h\|$ for all 
$h\in\T$. Assume that there exists an admissible sequence of future pointing curves $\gamma_n\colon [a_n,b_n]\to M$ 
with $\|\rho(\gamma_n)\|\le n^{-1}$. Partition $[a_n,b_n]$ into subintervals $[a_{n,i},b_{n,i}]$ such that 
$b_{n,i}-a_{n,i}\in [n,2n]$. We have 
$$\frac{1}{n}(b_n-a_n)\ge \|\gamma_n(b_n)-\gamma_n(a_n)\|\ge \e \sum_{i} \|\gamma_n(b_{n,i})-\gamma_n(a_{n,i})\|.$$
Since $b_n-a_n=\sum_i (b_{n,i}-a_{n,i})$ there exists an index $i$ with $\e\|\gamma_n(b_{n,i})-\gamma_n(a_{n,i})\|\le 2$. 
Consequently we have constructed an admissible sequence of future pointing curves $\gamma'_n\colon [a_n,b_n]\to M$ 
with $\|\gamma'_n(b_n)-\gamma'_n(a_n)\|\le 2\e^{-1}$.
By the previous arguments $\gamma'_n$ has to stay in a uniformly compact subset of $\overline{M}$. But this 
contradicts the strong causality of globally hyperbolic spacetimes.
\end{proof}

\section{Proof of theorem \ref{T16}}\label{rc}

Proposition \ref{C16} is a necessary ingredient in the proof of theorem \ref{T16}. 

\subsection{Proof of Proposition \ref{C16}}
For $p\in M$ let $\T_p$ be the set of classes $k\in H_1(M,\mathbb Z)_\R$ which contain a timelike future pointing
curve through $p$. $\T_p$ is obviously a positively homogenous subset of $\T \cap H_1(M,\mathbb Z)_{\R}$. 
A homology class $h\in H_1(M,\R)$ is called $\T_p$-rational if $nh\in \T_p$ for some $n\in\N$.
\begin{lemmaloc}
Let $(M,g)$ be compact and vicious. Then for every $p\in M$ the set of $\T_p$-rational homology classes is dense in $\T$. 
\end{lemmaloc}

\begin{proof}
Let $h\in\T$ and $\e>0$. To prove the lemma we have to find a $\T_p$-rational $h'\in \T\cap B_\e^{\|.\|}(h)$. 

Choose $\overline{p}\in\overline{\pi}^{-1}(p)$. With proposition \ref{P01a} we know that for all $h''\in \T$ there exists a $q''\in J^+(\overline{p})$ 
such that $\|h''-(q''-p)\|\le \err(g,g_R)$. Further we know with note \ref{F1} that there exists a $k''\in H_1(M,\Z)_\R$ such that $\overline{p}+k''\in I^+(q'')$ and 
$\dist(q'',\overline{p}+k'')\le \fil(g,g_R)$. Thus we have
\begin{align*}
\|h''-k''\|&\le \|h''-(q''-\overline{p})\|+\|q''-(\overline{p}+k'')\| \\
&\le \err(g,g_R)+\fil(g,g_R)+\std =:C.
\end{align*}
For $h''=n\cdot h$ with $n\in \N$, we obtain $\|h-k''/n\|\le C/n$. The claim follows for $n\e\ge C$.
\end{proof}

\begin{propositionloc}\label{mainc}
The set of $(\cap_{p\in M}\T_p)$-rational homology classes is dense in $\T$.
\end{propositionloc}

\begin{lemmaloc}
There exists $C<\infty$ such that for all $p,q \in M$ there are $\overline{p}\in \overline{\pi}^{-1}(p)$ and 
$\overline{q}\in\overline{\pi}^{-1}(q)$
with $\dist (\overline p,\overline q)<C$ and $\e_{p,q},\delta_{p,q}>0$, such that for all 
$\overline r \in B_{\e_{p,q}}(\overline p)$ and all $\overline s\in B_{\e_{p,q}}(\overline q)$, we have
$$B_{\delta_{p,q}}(\overline r)\subseteq I^-(\overline s), B_{\delta_{p,q}}(\overline s)
\subseteq I^+(\overline r).$$
\end{lemmaloc}

\begin{proof}
Choose any timelike future pointing curve $\gamma$ connecting $p$ with $q$ of $g_R$-length less than $\fil(g,g_R)$. 
Considering a lift of $\gamma$ to $\overline M$ with endpoints $\overline p$ resp. $\overline q$, the claim follows 
when considering normal neighborhoods around $\overline p$ and $\overline q$.

\end{proof}

\begin{proof}[Proof of Proposition \ref{mainc}]
The proposition follows from the observation that there exists a constant $C<\infty$ such that for every $x\in M$ and every $k_x\in \T_x$ there
exists a $k\in \cap_{p\in M} \T_p$ with $\|k-k_x\|\le C$. The proposition then follows since the $\T_x$-rational points are dense in $\T$ for every $x\in M$. 

To prove the observation we will construct a closed timelike curve of bounded Riemannian arclength that ``almost fills'' the manifold $M$.
Fix $x\in M$ and let $\e >0$ be the minimum of the Lebesgue numbers of the coverings 
$$\{B_{\e_{p,q}}(p)\times B_{\e_{p,q}}(q)\}_{p,q\in M}\text{ and }\{B_{\delta_{p,q}}(p)\times  B_{\delta_{p,q}}(q)\}_{p,q\in M}$$ 
of $M\times  M$. Then for all $p,q\in M$ there exist $\overline p\in \overline{\pi}^{-1}(p)$ and $\overline q\in
\overline{\pi}^{-1}(q)$ with $\dist (\overline p,\overline q)\le \fil(g,g_R)$ such that 
   $$B_{\e}(\overline r)\subseteq I^-(\overline s), B_{\e}(\overline s)\subseteq 
   I^+(\overline r)$$
for all $\overline r \in B_{\e}(\overline p)$ and all $\overline s\in B_{\e}(\overline q)$.
Take a finite subcover $\{B_{\e}(p_1),\ldots ,B_{\e}(p_N)\}$ of $M$ and choose timelike future pointing 
curves $c_1\colon [0,N]\to M$, with $c_1(n)=p_{N-n}$ for $0\leq n\leq N-1$ and $c_1(N)=x$,
and $c_2\colon [0,N]\to M$, with $c_2(0)=x$ and $c_2(n)=p_n$ for $1\leq n\leq N$ such that 
for one (hence every) lift $\overline c_1$ resp. $\overline c_2$ of $c_1$ resp. $c_2$ we have
$B_\e(\overline c_{i}(n+1))\subseteq I^+(\overline c_{i}(n))$ $(i=1,2)$. The $g_R$-arclength of both 
curves can be bounded by $(N+1)\fil(g,g_R)$.  
By joining a timelike future pointing representative of $k\in \T_x$ with $c_1$ and $c_2$, we obtain 
$k+[c_1\ast c_2]\in \cap_{1\leq i\leq N}\T_{p_i}$. The observation follows if we can show that  $k+[c_1\ast c_2]\in 
\cap_{p\in M}\T_{p}$. This can be seen as follows: For $y\in M$ choose $p_i$ with $y\in B_\e(p_i)$. Let 
$\overline y$ be a lift of $y$ to $\overline M$, $\overline p_i$ a lift of $p_i$ with $\overline y\in 
B_\e(\overline p_i)$ and $\overline{c_1*c_2}$ a lift of $c_1*c_2$ through $\overline p_i$. We can choose a timelike 
future pointing curve $\beta_1$ from $\overline{c_1*c_2}(i-1)$ to $\overline{c_1*c_2}(i+1)$ via $\overline y$ and 
homotopic to $\overline{c_1*c_2}|_{[i-1,i+1]}$ if $\overline p_i=\overline{c_1*c_2}(i)$. In the same manner choose 
a future pointing curve $\beta_2$ from $\overline{c_1*c_2}(2N-i-1)$ to $\overline{c_1*c_2}(2N-i+1)$ via $\overline y$ 
if $\overline p_i=\overline{c_1*c_2}(2N-i)$.
If we substitute $\overline{c_1*c_2}|_{[i-1,i+1]}$ with $\beta_1$ and $\overline{c_1*c_2}|_{[2N-i-1,2N-i+1]}$ 
with $\beta_2$, we obtain a timelike future pointing curve homologous to $c_1*c_2$. Thus $k+[c_1\ast c_2]\in \T_y$.

\end{proof}

\begin{lemmaloc}\label{1.5} 
Let $(M,g)$ be compact and vicious. Then there exists $C<\infty$ such that 
for every future pointing curve $\gamma\colon [a,b]\to \overline M$ there exists $k\in \cap_{p\in M}\T_{p}$ with 
$\|\gamma(b)-\gamma(a)-k\|\le C$ and $\e_k>0$, such that  $B_{n\e_k}(p+nk)\subseteq I^+(p)$ 
for all $p \in \overline M$ and all $n\in \mathbb N$.
\end{lemmaloc}

\begin{proof}
The same argument used in the proof of proposition \ref{mainc} shows: There exists $\e_k>0$ such that  
$B_{\e_k}(p+k)\subseteq I^+(p)$ for all $p\in\overline M$. The claim then follows inductively.

\end{proof}

\begin{proof}[Proof of Proposition \ref{C16}]
The proof is an easy consequence of lemma \ref{1.5}. Take any $k\in \cap_{p\in M}\T_{p}$ and $n\in \N$ such that 
$n\e_k\ge \diam(M,g_R)$. Then $B_{\diam(M,g_R)}(\overline{p}+nk)\subset I^+(\overline{p})$ for all $\overline{p}\in
\overline{M}$. This implies directly that $\cap_{p\in M}\T_{p}$ contains a basis of $H_1(M,\R)$. 

\end{proof}

\subsection{Proof of Theorem \ref{T16}}

\begin{propositionloc}\label{3.2}
For every $R>0$ there exists a constant $0<K=K(R)<\infty$ such that
    $$B_R(q)\subseteq I^+(p)$$
for all $p,q\in \overline{M}$ with $q-p\in\T$ and $\dist_{\|.\|}(q-p,\partial \T)\ge  K$.
\end{propositionloc}

Note that there exists a $K<\infty$ such that for every $p\in\overline M$ the intersection $B_K( p)\cap I^+(p)$ 
contains a fundamental domain of the Abelian covering $\overline{\pi}\colon\overline M \to M$.

\begin{proof}
Choose a basis $\{k_1,\ldots,k_b\}\subset \cap_{p\in M}\T_p$ of $H_1(M,\R)$
such that there exists an $\e_0>0$ with $B_{\e_0}(q+k_i)\subseteq J^+(q)$ for all $q\in \overline M$ and all 
$1\leq i\leq b$. The existence of the $k_i$ is ensured by lemma \ref{1.5}.

Set $K':=\sup_{p \in \overline M}\sup_{1\leq i\leq b}\dist(p,p +k_i)$ and 
$K'':= (\frac{R+b K'}{\e_0}+b)(K'+\std)$. For $h=\sum r^ik_i\in H_1(M,\R)$ with $r^i\ge 0$ and $\|h\|>K''$ 
we have
$$\sum r^i\geq \frac{R+b K'}{\e_0}+b.$$
Because of $\sum [r^i]\geq \sum r^i -b$ we obtain $\sum [r^i]\geq \frac{R+b K'}{\e_0}$. By the choice 
of $K'$ and $\e_0$ we conclude
   $$B_R(x+h)\subseteq B_{R+b K'}(x+\sum [r^i]k_i)\subseteq B_{\sum[r^i]\e_0}(x+
     \sum[r^i]k_i)\subseteq I^+(x)$$
with lemma \ref{1.5} for every point $x\in \overline M$. Now if we have
$$\dist\nolimits_{\|.\|}(q-p,\partial \T)\ge K''+\err(g,g_R)+\std=:K$$ 
there exists $r\in I^+(p)$ with $q -r\in \pos\{k_1,\ldots ,k_b\}$ and 
$\|q-r\|\ge K''$ (proposition \ref{P01a}). Since $B_R(q)\subset I^+(r)$ we conclude
  $$B_R(q)\subset I^+(p).$$
  
\end{proof}

\begin{remark}\label{r11}
Lemma \ref{1.5} implies: For every $\e >0$ there exist $N(\e)\in \mathbb N$ and $k_1,\ldots ,k_N\in \cap_{p\in M}\T_p$ with $\T_\e \subseteq 
\pos\{k_1,\ldots ,k_N\}$. Since $\T^\circ \neq \emptyset$ we know that for $\e>0$ sufficiently small, $\{k_1,\ldots ,k_N\}$ necessarily contains a basis of 
$H_1(M,\R)$.
\end{remark}

Recall from proposition \ref{P01a} that $\dist_{\|.\|}(\T,J^+(x)-x)\le \err(g,g_R)$ for all $x\in \overline{M}$. Further recall that $\T$ is a compact cone. Consequently 
we can choose $0<\delta <1$ and $K_0<\infty$ such that 
   $$\|\sum h_i\|\ge \delta \sum \|h_i\|$$
for any finite set $\{h_i\}_{i=1,\ldots n}\subseteq B_{\err(g,g_R)}(\T)\setminus B_{K_0}(0)$.

\begin{lemmaloc}\label{F15}
Set $K_1(\e):=\max\{K_0,\frac{4b\err(g,g_R)}{\delta\e}\}$ and let $\e>0$ be given. Further let $\{h_i\}_{1\le i\le N}\subset \T$ with $\|h_i\|\ge K_1$, $\frac{1}{2}\le 
\frac{\|h_i\|}{\|h_j\|}\le 2$ and $\sum h_i \in \T_\e$ for all $1\le i,j\le N$. Then there exists a subset $\{i_1,\ldots ,i_b\}\subset \{1,\ldots ,N\}$ with $\sum_j h_{i_j}\in 
\T_{\eta}$ for $\eta:=\frac{\delta}{8b}\e$.
\end{lemmaloc}

\begin{proof}
The assumption $\sum h_i \in \T_\e$ implies that 
$$\conv\{h_1,\ldots ,h_N\}\cap \T_\e \neq \emptyset .$$
With the theorem of Caratheodory follows: There exist $1\le i_1,\ldots ,i_b\le N$ and 
$\lambda_1,\ldots ,\lambda_b\ge 0$ with $\sum_j \lambda_j =1$ such that $\sum \lambda_{j}h_{i_j}\in \T_\e$. 
For every $\alpha\in\T^\ast$ with $\|\alpha\|^\ast=1$ and every $j\in \{1,\ldots ,b\}$ we have
\begin{align*}
\max_m\{\alpha(h_{i_m})\}\ge \alpha(\sum_m \lambda_m h_{i_m})\ge \e \|\sum_m \lambda_m h_{i_m}\|
\ge \delta \e \sum_m \lambda_m\|h_{i_m}\|\ge \frac{\delta}{2}\e \|h_{i_j}\|.
\end{align*}
And therefore for $\alpha(h_{i_k})=\max_m\{\alpha(h_{i_m})\}$ we get
\begin{align*}
\frac{1}{\|\sum_j h_{i_j}\|}\alpha(\sum_j h_{i_j})&\ge \sum_j \frac{1}{2b \|h_{i_j}\|} \alpha(h_{i_j})\\
&\ge \frac{1}{2b\|h_{i_k}\|} \alpha(h_{i_k})-\sum_j\frac{1}{2b\|h_{i_j}\|}\err(g,g_R)\\
     &\ge \frac{\delta}{4b}\e-\frac{\err(g,g_R)}{2K_1}\ge \frac{\delta}{8b}\e.
\end{align*}

\end{proof}

\begin{proof}[Proof of Theorem \ref{T16}]
(i) First we reduce the claim to the following special case: {\it For every $\e>0$ there exists $C(\e)<\infty$ 
such that $|d(x,y)-d(z,w)|\le C(\e)$ for all $x,y,z,w\in \overline{M}$ with $y-x,w-z\in \T_{\e}$ and }
$$\dist(x,z),\dist(y,w)< K_2:=\max\{\fil(g,g_R),2\}+\std.$$
Let $x,y,z,w\in \overline{M}$ be given with $y-x,w-z\in \T_\e$. Choose $k_{x,z}\in H_1(M,\Z)_\R$ with 
$\dist(x-k_{x,z},z)\le \diam(M,g_R)$. For every $k\in H_1(M,\Z)_\R$ we have
\begin{align*}
\dist(x+k,z)+\dist(y+k,w)&\ge \|(x-z)-(y-w)\|-2\std\\
&\ge \|(y-w)-k_{x,z}\|-\diam(M,g_R)-3\std\\
&\ge \dist(y-k_{x,z},w)-\diam(M,g_R)-4\std.
\end{align*}
Since we have $d(x+k,y+k)=d(x,y)$ for every $k\in H_1(M,\Z)_\R$, we can assume that 
$\dist(x,z)< \fil(g,g_R)$. Note that $\diam(M,g_R)\le \fil(g,g_R)$.

If we have $\|y-x\|\ge \frac{2+\e}{\e}K_2\ge \frac{2+\e}{\e}\|x-z\|$ then
\begin{align*}
\dist\nolimits_{\|.\|}(y-z,\partial\T)\ge \dist\nolimits_{\|.\|}(y-x,\partial\T)-\|x-z\|\ge \frac{\e}{2}\|y-z\|.
\end{align*}
The special case then yields $|d(x,y)-d(z,y)|\le C(\frac{\e}{2})$. 

For any integer $1\le i\le n:=[\|y-w\|]$ set 
$h_i:= (w-z)+\frac{i}{n}(y-w)$. Since $\T_{\e/2}$ is convex we have $h_i\in \T_{\e/2}$ for $1\le i\le n$.
Choose points $w_i\in \overline{M}$ with $w_i-z=h_i$ for $1\le i\le n$. With the special case we have 
$(\text{Note that }\dist(w_i,w_{i+1})\le 2+\std)$
$$|d(z,w)-d(z,w_1)|,|d(z,w_i)-d(z,w_{i+1})|,|d(z,w_n)-d(z,y)|\le C(\e/2)$$
for all $1\le i\le n-1$. With the triangle inequality we get
\begin{align*}
|d(x,y)-d(z,w)|&\le (n+2)C(\e/2)\\
&\le C(\e/2)(\std+2)(\dist(x,z)+\dist(y,w)+1)\\
&=:L_c(\e)(\dist(x,z)+\dist(y,w)+1).
\end{align*}
The case $\|y-x\|< \frac{2+\e}{\e}K_2$ can be absorbed into the constant $L_c(\e)$ since the time separation is 
bounded on any compact subset of $\overline{M}\times\overline{M}$. This shows that the general claim follows 
from the special case.

(ii) The special case follows from 
\begin{equation}\label{E20}
d(x,y)\ge d(z,w)-C(\e)
\end{equation}
for $x,y,z,w\in \overline{M}$ with $\dist(x.z),\dist(y,w)< K_2$ and $y-x,w-z\in \T_{\e}$.
Exchanging $(x,y)$ and $(z,w)$ in (\ref{E20}) we get $d(z,w)\ge d(x,y)-C(\e)$. Consequently we have 
$$|d(x,y)-d(z,w)|\le C(\e)$$
and with it the special case.

Recall the definition of $K(.)$ from proposition \ref{3.2}. Set 
$$K_3:=\max\left\{\frac{1}{\eta \delta b}\left(K(K_2)+b(\fil(g,g_R)+\std)\right), K_1\right\}.$$
To prove (\ref{E20}) we can assume that 
$$z\in J^+(x)\text{ and }\dist(z,w)\ge 2\delta b K_3.$$  
By note \ref{F1} there exists $k\in H_1(M,\Z)_\R$ with $\dist(x,z+k)< \fil(g,g_R)$ and $z+k\in J^+(x)$.
The case $\dist(z,w)< 2\delta b K_3$ can be absorbed into the constant $C(\e)$ since the time separation 
is bounded on compact subsets of $\overline{M}\times \overline{M}$.

Choose a maximal future pointing curve $\gamma\colon [0,T]\rightarrow \overline{M}$ connecting 
$z$ with $w$. With our assumption that 
$\dist(z,w)\ge 2\delta b K_3$, we can partition $[0,T]$ into at least $b$-many mutually disjoint intervals $[s_i,t_i]$ with 
$K_3\le \|\gamma(t_i)-\gamma(s_i)\|\le 2K_3$. Then by lemma \ref{F15} there exist intervals 
$[s_{m_j},t_{m_j}]\subset [0,T]$ $(1\le j\le b)$ with $\sum \gamma(t_{m_{j}})-\gamma(s_{m_{j}})\in \T_{\eta}$ 
$(\eta:=\frac{\delta}{8b}\e)$. 
After relabeling we can assume $t_{m_{i}}\le s_{m_{i+1}}$. We want to build a future pointing curve from $z$ to $y$ using pieces of 
$\gamma$. Choose $k_i\in H_1(M,\Z)_\R$ such that 
$$\gamma(t_{m_i})+k_i\in J^+(\gamma(s_{m_i}))\cap B_{\fil(g,g_R)}(\gamma(s_{m_i}))$$
and future pointing curves $\zeta_i\colon [s_{m_i},t_{m_i}]\to \overline{M}$ from $\gamma(s_{m_i})$ to 
$\gamma(t_{m_i})+k_i$. Define the future pointing curve $\gamma'\colon [0,T]\to \overline{M}$ as 
$$\gamma':=\gamma|_{[0,s_{m_1}]}\ast \zeta_1\ast (\gamma|_{[t_{m_1},s_{m_2}]}+k_1)\ast \zeta_2 \ast \ldots \ast 
\left(\gamma|_{[t_{m_b},T]}+\sum_{i=1}^b k_i\right).$$
Set $h'_i:=\gamma(t_{m_i})-\gamma(s_{m_i})$ and $l_i:=\zeta(t_{m_i})-\zeta(s_{m_i})$.
By construction  we have 
$$w-\gamma'(T)=\sum_{i=1}^b h'_i-l_i.$$
Note that $\sum \|l_i\|\le b(\fil(g,g_R)+\std)$. We have
\begin{align*}
\dist\nolimits_{\|.\|}(w-\gamma'(T),\partial \T)&\ge \dist\nolimits_{\|.\|}\left(\sum h'_i,\partial\T\right)-
\sum \|l_i\|\\
&\ge \eta \|\sum h'_i\| -\sum \|l_i\|\ge \eta \delta \sum\|h'_i\|-\sum\|l_i\|\\
&\ge \eta \delta b K_3-b(\fil(g,g_R)+\std)\\
&\ge K(K_2)>0.
\end{align*}
Since $\sum h'_i\in \T$ and $\dist_{\|.\|}(\sum h'_i-l_i,\partial \T)>0$, we get $w-\gamma'(T)\in \T$. With 
proposition \ref{3.2} we have $y\in B_{K_2}(w)\subset I^+(\gamma'(T))$.
Therefore we can choose future pointing curves $\zeta_{0}$ and $\zeta_{b+1}$ connecting $x$ with $z$ resp. 
$\gamma'(T)$ with $y$ and obtain 
$$d(x,y)\ge L^{\overline{g}}(\zeta_0\ast \gamma'\ast \zeta_{b+1})\ge L^{\overline{g}}(\gamma)
-\sum_{i=1}^b L^{\overline{g}}\left(\gamma|_{[s_{m_i},t_{m_i}]}\right).$$
Choose $\Lambda_{g,g_R}<\infty$ such that $|g(v,v)|\le \Lambda_{g,g_R} g_R(v,v)$ for all $v\in TM$. 
With corollary \ref{1.10} we  have
\begin{align*}
L^{\overline{g}}(\gamma|_{[s_{m_i},t_{m_i}]})&\le \Lambda_{g,g_R} C_{g,g_R} \dist(\gamma(s_{m_i}),\gamma(t_{m_i}))\\
&\le \Lambda_{g,g_R} C_{g,g_R}(2K_3+\std).
\end{align*}
Therefore we conclude
$$d(x,y)\ge d(z,w)- \Lambda_{g,g_R} C_{g,g_R} b(2K_3+\std)=:d(z,w)-C(\e).$$

\end{proof}

\appendix
\section{Sullivan's structure cycles}\label{cones}

In this section we first briefly recall the main definitions and results for structure currents and structure cycles of a cone structure $C$ from \cite{sul}. At the end we 
explain (proposition \ref{T=C}) the relation between the homology classes of structure cycles and the stable time cone $\T$ for the cone structure of future pointing 
vectors in a compact spacetime.

\begin{definitionloc}[\cite{sul}, definition I.1, I.2]
\begin{enumerate}
\item A {\it compact convex cone} $C$ in a (locally convex topological) vector space over $\R$ is a  convex cone which for some (continuous) linear functional $L$ 
satisfies $L(x)>0$ for $x\neq 0$ in $C$ and $L^{-1}(L(x))\cap C$ is compact.
\item A {\it cone structure} on a closed subset $X$ of a smooth manifold $M$ is a continuous field of compact convex cones $\{C_x\}$ in the vector spaces 
$\Lambda_p M_x$ of tangent $p$-vectors on $M$, $x\in X$. Continuity of cones is defined by the Hausdorff metric on the compact subsets of the rays in 
$\Lambda_pM$. 
\end{enumerate}
\end{definitionloc}

Obviously the set of future pointing tangent vectors in a time oriented Lorentzian manifold is an example of a cone structure. This connection is discussed briefly in 
\cite{sul} p. 248/249 $(p=1)$. Other examples include the oriented tangent $p$-planes to an oriented $p$-dimensional foliation $(1\le p\le m)$. 

\begin{definitionloc}[\cite{sul}, definition I.3]
A smooth differential $p$-form $\omega$ on $M$ is {\it transversal} to the cone structure $C$ if $\omega(v)>0$ for each $v\neq 0$ in $C_x$, $x\in M$. 
\end{definitionloc}

\begin{definitionloc}[\cite{sul}, definition I.4]
A {\it Dirac current} is one determined by the evaluation of $p$-forms on a single $p$-vector at one point. The {\it cone of structure currents} $\mathcal{C}$ 
associated to the cone structure $C$ is the closed convex cone of currents generated by the Dirac currents associated to elements of $C_x$, $x\in M$.
\end{definitionloc}

\begin{propositionloc}[\cite{sul}, proposition I.4, I.5]
\begin{enumerate}
\item A cone structure $C$ admits transversal $1$-forms.
\item Let $X$ be a compact subset of $M$. Then the cone of structure currents $\mathcal{C}$ associated to a cone structure $C$ on $X$ is a compact convex cone.
\end{enumerate}
\end{propositionloc}

\begin{definitionloc}[\cite{sul}, definition I.6]
If $C$ is a cone structure, the {\it structure cycles} of $C$ are the structure currents which are closed as currents.
\end{definitionloc}

\begin{theoremloc}[\cite{sul}, theorem I.7]\label{TS10}
Suppose $C$ is a cone structure of $p$-vectors defined on a compact subspace $X$ in the interior of $M$ which is also compact
(with or without boundary).
\begin{itemize}
\item[(i)] There are always non-trivial structure cycles in $X$ or closed $p$-forms on $M$ transversal to the cone structure. 
\item[(ii)] If no closed transverse form exists some no-trivial structure cycle in $X$ is homologous to zero in $M$.
\item[(iii)] If no non-trivial structure cycle exists some transversal closed form is cohomologous to zero.
\item[(iv)] If there are both structures cycles and transversal closed forms then
\begin{itemize}
\item[(a)] the natural map
$$\text{(structure cycles on $X\rightarrow$ homology classes in $M$)}$$
is proper and the image is a compact cone $\C\subset H_p(M,\R)$
\item[(b)] the interior of the dual cone $\C'\subset H^p(M,\R)$ consists precisely of the classes of closed forms transverse to $C$.
\end{itemize}
\end{itemize}
\end{theoremloc}

\begin{propositionloc}[\cite{sul}, proposition I.8]
Any structure current $c$ may be represented $c=\int_X v d\mu$ where $\mu$ is a non-negative measure on $X$ (assume compact) and $v$ is a $\mu$-integrable
function into $p$-vectors satisfying $v(x)\in C_x$.
\end{propositionloc}

Let $M$ be compact and oriented. Then for $p+q=m=\dim M$ the currents in the image of the map $\{$smooth $q$-forms$\} \to \{p$-currents$\}$, $\tau\mapsto 
(\omega\mapsto \int \omega\wedge \tau)$ are called the {\it diffuse} $p$-currents (\cite{sul} p.232). One can then define operators $D_\e\colon \{p\text{-currents}\}\to 
\{\text{diffuse }p\text{-currents}\}$ $(\e>0)$ such that 
\begin{itemize}
\item[(i)] $D_\e c$ approaches $c$ for $\e\to 0$ for all $p$-currents $c$.
\item[(ii)] $D_\e$ commutes with the boundary operator $\partial$ and preserves homology classes (\cite{sul} p. 232-234).
\end{itemize}
Note that we can represent any diffuse $p$-current $c$ by a smooth measure $\mu$, i.e. a smooth $m$-form, and a smooth $p$-vector field $V$, i.e.
$c=\int Vd\mu$ (\cite{sul} \textsection 2).

A {\it foliation cycle} to an oriented foliation is a structure cycle for the cone structure defined by the oriented tangent spaces to the foliation (\cite{sul} p. 227).

Consider a non-zero vector field $X$ with flow $\phi_t$ tangent to a $1$-dimensional foliation $\mathfrak{F}$. Then $X$ defines a map between measures and 
structure $1$-currents, $\mu\to (X,\mu)$ with $(X,\mu)(\omega)=\int \omega(X)d\mu$.

\begin{propositionloc}[\cite{sul}, proposition II.24]\label{II.24}
$\mu\to (X,\mu)$ defines continuous bijections
\begin{itemize}
\item[(i)] non negative measures on $M$ $\sim$ foliation currents along the foliation,
\item[(ii)] measures invariant under flow $\sim$ foliation cycles along the foliation.
\end{itemize}
\end{propositionloc}

For the rest of this section we will focus on the cone structure of future pointing causal tangent vectors of a spacetime. It is obvious that the set of future pointing 
vectors in a spacetime form a cone structure in the above sense.

\begin{lemmaloc}\label{L101}
Let $(M.g)$ be compact and oriented spacetime. Then every structure cycle can be approximated by foliation cycles of oriented timelike $1$-dimensional foliations. 
Here oriented means oriented by a future pointing vector field. 
\end{lemmaloc}

\begin{proof}
Let $c$ be a structure cycle. For $x\in M$ we can choose $\e_x >0$ and a future pointing timelike vector field $V_x$ such that all orbits of $V_x$ through a 
neighborhood $U_x$ of $x$ are closed. Next choose a smooth Borel measure $\mu_x$ invariant under the flow of $V_x$ and $\supp \mu_x\subseteq U_x$. Then 
the current $c_x:=\int_M V_x d\mu_x$ is a structure cycle. 

Choose a finite cover $\{U_{x_i}\}_{1\le i\le N}$ of $M$ by such neighborhoods and consider the structure cycle $c_\delta:=c+\delta\sum_{i} c_{x_i}$. Clearly any 
Borel measurable vector field representing $c_\delta$ is future pointing timelike everywhere. Thus any smooth vector field $V_{\delta,\e}$ representing 
$D_\e c_\delta$ will be future pointing timelike for $\e>0$ sufficiently small. 

Clearly $D_\e c_\delta$ is a foliation cycle of the $1$-dimensional foliation determined by $V_{\delta,\e}$. Since $D_\e c_\delta$ approaches $c_\delta$ for 
$\e\to 0$ and $c_\delta\to c$ for $\delta \to 0$, the claim follows.

\end{proof}

\begin{propositionloc}\label{T=C}
Let $(M,g)$ be a compact and vicious spacetime. Then the cone $\C$ of homology classes of structure cycles coincides with the stable time cone $\T$ (see section 
\ref{sec2} for the definition of $\T$). 
\end{propositionloc}

\begin{proof}

(i) $\T\subseteq \C$: By the definition of $\T$ there exists an admissible sequence of future pointing curves $\gamma_n \colon [a_n,b_n]\to M$, 
parameterized by $g_R$-arclength and  $C\ge 0$ such that $C\rho(\gamma_n)\to h$ for $n\to\infty$ for every $h\in \T$. Without loss we can assume that the 
$\gamma_n$ are smooth. Define the structure current $c_n(\omega):=\frac{C}{b_n-a_n} \int_{\gamma_n} \omega$ for a $1$-form $\omega$. 

The norm of $c_n$ is bounded by $C<\infty$. Therefore there exists a subsequence $\{c_{n_k}\}_{k\in\N}$ convergent to a structure current $c$ in the 
weak-$\ast$-topology on currents. Since $b_n-a_n\to \infty$ any accumulation point $c$ of $\{c_n\}_{n\in\N}$ vanishes on exact forms. Consequently $c$ is a 
structure cycle with homology class $h$. 

(ii) $\C\subseteq \T$: We can assume without loss of generality that $M$ is oriented. Otherwise consider the (twofold) orientation cover  $M^{or}$ of $M$. The 
canonical projection $\pi^{or}\colon M^{or} \to M$ induces an isomorphism $\pi^{or}_\ast\colon H_1(M^{or},\R)\to H_1(M,\R)$. Lift the Lorentzian metric $g$ to the
Lorentzian metric $g^{or}$ on $M^{or}$. Note that $g^{or}$ is time-orientable as well. Choose the time-orientation compatible with $\pi^{or}$. Consider on the 
orientation cover the cone structure of future pointing causal vectors. Then it is clear that the cone of homology classes of structure cycles in $H_1(M^{or},\R)$ is 
mapped onto $\C\subseteq H_1(M,\R)$ by $\pi_\ast^{or}$. In the same way it is clear that the stable time cone of $(M^{or},g^{or})$ is mapped onto $\T$ by 
$\pi^{or}_\ast$. Thus if we prove $\C\subseteq \T$ in the orientable case, the non-orientable case follows as well. Further since $M^{or}$ is a finite covering, 
$(M^{or},g^{or})$ is vicious if $(M,g)$ is vicious.

Let $c$ be a structure cycle. Choose an approximation of $c$ by a foliation cycle $\hat{c}$ of an oriented $1$-dimensional timelike foliation $\mathfrak{F}$ 
according to lemma \ref{L101}. Choose a future pointing (timelike) vector field $X$ tangent to $\mathfrak{F}$ with $|X|\equiv 1$ and a Borel measure $\mu$ 
invariant under the flow $\phi_t$ of $X$ such that $\hat{c}=\int X d\mu$ (proposition \ref{II.24}). By the theorem of Krein-Milman any invariant measure is 
approximated in the weak-$\ast$-topology on currents by a positive combination of $\phi_t$-ergodic probability measures. 

Consider the measure $\sum_{i=1}^N \lambda_i \mu_i$, where $\lambda_i>0$ and $\mu_i$ are $\phi_t$-ergodic probability measures. For every $1\le i\le N$ 
choose an $\mu_i$-generic $x_i\in M$, i.e. 
\begin{equation}\label{E200}
\frac{1}{2T} \phi_{.}(x_i)_\sharp(\mathcal{L}^1|_{[-T,T]})\stackrel{\ast}{\rightharpoonup}\mu_i
\end{equation}
for $T\to\infty$. 

Denote with $h_i$ the homology class of the foliation cycle $\int X d\mu_i$. Then \eqref{E200} implies $\rho(\phi_{.}(x_i)|_{[-T,T]})\to h_i$ for $T\to\infty$. 
Since $X$ is timelike this implies $h_i\in \T^1$. Further since $\sum_{i=1}^N \lambda_i \mu_i$ approximates $\mu$, $\sum_i \lambda_i h_i$ approximates the 
homology class of $\hat{c}$. The property that $\T$ is closed then implies that the homology class of $\hat{c}$ is contained in the stable time cone. The same 
argument then shows that the homology class of $c$ is in $\T$. 

\end{proof}

\end{document}